\documentclass[pdflatex,sn-mathphys-num]{sn-jnl}

\usepackage{graphicx}%
\usepackage{multirow}%
\usepackage{amsmath,amssymb,amsfonts}%
\usepackage{amsthm}%
\usepackage{mathrsfs}%
\usepackage[title]{appendix}%
\usepackage{xcolor}%
\usepackage{textcomp}%
\usepackage{manyfoot}%
\usepackage{booktabs}%
\usepackage[ruled]{algorithm2e}
\usepackage{algpseudocode}%
\usepackage{listings}%
\usepackage{subfigure}
\usepackage{booktabs}
\usepackage{enumitem}
\usepackage{algpseudocode}
\usepackage{tikz}
\usetikzlibrary{arrows.meta}

\theoremstyle{thmstyleone}%
\newtheorem{lemma}{Lemma}
\newtheorem{theorem}{Theorem}
\theoremstyle{thmstyletwo}%
\newtheorem{example}{Example}%
\newtheorem{remark}{Remark}%

\theoremstyle{thmstylethree}%
\newtheorem{assumption}{Assumption}
\begin{document}
	
	\title[Article Title]{An Alternating Minimization Algorithm for\\ Federated Sparse Orthogonal Nonnegative\\ Matrix Factorization}
	
	
	\author[1]{\fnm{Yanjiao} \sur{Zhu}}\email{zhuyj87@mail.sysu.edu.cn}
	
	\author[2]{\fnm{Xianchao} \sur{Xiu}}\email{xcxiu@shu.edu.cn}

	\author[1]{\fnm{Yongyang} \sur{Xiong}}\email{xiongyy25@mail.sysu.edu.cn}
	
	\author*[1]{\fnm{Wanquan} \sur{Liu}}\email{liuwq63@mail.sysu.edu.cn}
	
	\affil*[1]{\orgdiv{School of Intelligent Systems Engineering}, \orgname{Sun Yat-sen University}, \orgaddress{ \city{Guangzhou}, \postcode{510275},  \country{China}}}
	
	\affil[2]{\orgdiv{School of Mechatronic Engineering and Automation}, \orgname{Shanghai University}, \orgaddress{ \city{Shanghai}, \postcode{200444}, \country{China}}}

	
	\abstract{In this paper, we develop an alternating minimization algorithm for federated sparse orthogonal nonnegative matrix factorization (ONMF) with emphases on controlling client drift. To reduce communication frequency, clients perform multiple local updates between communication rounds, biasing updates toward their respective local objectives and giving rise to client drift, particularly under data heterogeneity. In federated sparse ONMF, global nonsmooth regularization further complicates drift control, as proximal mappings generally do not commute with averaging. To address these challenges, we develop an effective optimization algorithm that introduces variable splitting to decouple proximal operator evaluation from communication, allowing clients to handle joint sparsity regularization locally. In addition, the algorithm incorporates a correction term into the local gradients to adjust the update directions and mitigate client drift. We theoretically establish that the proposed algorithm achieves sublinear convergence to a bounded residual error under some mild conditions. Numerical experiments demonstrate favorable convergence behavior on a synthetic dataset and competitive clustering performance on real-world multi-view datasets.}

	
	
	
	\maketitle
	
	\section{Introduction}\label{sec1}

	Nonnegative matrix factorization (NMF) is a powerful dimensionality reduction technique that approximates a nonnegative data matrix by the product of two low-dimensional nonnegative matrices \cite{Saberi-Movahed2025,takahashi2026majorization}. It has broad applications in data clustering \cite{Peng2023}, fault detection \cite{zhang2023structured}, and gene expression classification \cite{Min2023}, and so on. Orthogonal nonnegative matrix factorization (ONMF) is an important extension of NMF and is closely related to $k$-means clustering~\cite{ding2006orthogonal,chen2026orthogonal}. Given a nonnegative data matrix $X\in\mathbb{R}_+^{d\times n}$, ONMF imposes an orthogonality constraint on the nonnegative basis matrix and is formulated as
	\begin{equation}\label{onmf}
		\begin{aligned}
			\min_{U,V}&\quad\|X-UV\|_F^2 \\
			\text{s.t.}&\quad U\ge 0,~ U^{\top} U = I,~ V\ge 0,  
		\end{aligned}
	\end{equation}
	where $U\in\mathbb{R}_+^{d\times r}$ is the basis matrix, $V\in\mathbb{R}_+^{r\times n}$ is the coefficient matrix, $I$ is the identity matrix with a proper dimension, and $r< n$ is the latent dimension.

\subsection{Federated Setting}\label{subsec11}	
	Scaling ONMF to large datasets distributed across multiple clients requires methods that coordinate computation while retaining raw data locally. Federated learning provides a framework for this purpose, allowing a central server to coordinate model training across clients with distributed data \cite{gobet2026federated,wang2024seeking}. 
	
	In this paper, we propose a federated sparse ONMF model involving $m$ clients, each holding a different feature block of the same samples. Specifically, client $p=1,\cdots,m$ has a local data $X_p\in\mathbb{R}_+^{d_p\times n}$ corresponding to the same $n$ samples. Each client maintains a local basis matrix $U_p\in\mathbb{R}_+^{d_p\times r}$, and $V\in\mathbb{R}_+^{r\times n}$ is the shared matrix. It introduces an $\ell_{2,1}$-norm on $V$ to promote joint sparsity, which helps reduce redundancy in the shared coefficient. The resulting optimization model is
	\begin{equation}\label{FMF}
		\begin{aligned}
			\min_{\{U_{p}\},V}&\quad\frac{1}{m}\sum_{p=1}^m\| X_{p} - U_{p} V \|_{F}^{2}+\mu\|V\|_{2,1} \\
			\text{s.t.}&\quad
			U_{p} \ge 0,~U_{p}^{\top} U_{p} = I,~V \ge 0,
		\end{aligned}
	\end{equation}
	where $\mu>0$ is the regularization parameter and $\|V\|_{2,1}=\sum_{i=1}^{r}\|V_{i .}\|_2$ with $V_{i .}$ denoting the $i$th row of $V$.
    
A natural idea is to adopt the standard federated averaging (FedAvg) \cite{pmlr-v54-mcmahan17a} framework for solving Problem~\eqref{FMF} in the federated setting. FedAvg reduces communication frequency by allowing each participating client to perform multiple local stochastic gradient descent steps between communication rounds. Applying this strategy of multiple local updates to Problem~\eqref{FMF} requires updating the two matrix variables repeatedly on each client before aggregation. These local optimization requirements motivate a review of existing ONMF methods, followed by a discussion of federated NMF methods.
	
\subsection{Overview of Related Works}\label{subsec12}	
	For standard NMF, fixing either matrix yields a least squares subproblem, which motivates alternating optimization methods using projected gradients \cite{LIN20119997}, active set solvers \cite{Zhang2014}, and block coordinate updates \cite{Fu2019}. For ONMF, nonnegativity and orthogonality constraints yield a nonconvex feasible set, making global optimization difficult. Various methods have been developed for ONMF and variants, including multiplicative updates \cite{ding2006orthogonal,Choi2008}, convergent additive updates for an orthogonality penalized formulation \cite{MIRZAL2014149}, hierarchical alternating least squares \cite{pmlr-v39-kimura14}, greedy orthogonal pivoting \cite{zhang2019greedy}, and subspace exploration \cite{NIPS2015_eae27d77}. Developing algorithms for sparse ONMF requires handling both nonnegative orthogonality constraints and joint sparsity regularization. Existing methods to nonnegative orthogonality constraints include augmented Lagrangian \cite{pompili2014two} and exact penalty methods \cite{jiang2023exact}. 
    Joint sparsity has been studied using proximal gradient methods \cite{beckhallak2017} and group cardinality approximations \cite{panchen2021}. Iterative hard thresholding methods address sparsity under nonnegativity constraints \cite{pan2017,zhang2024matrix}, while sparse ONMF methods include hierarchical alternating least squares and block proximal linearization \cite{LI2020112785}.
   
	Federated learning is often modeled as a distributed optimization problem \cite{beznosikov2026exploring,stanko2026accelerated}. Existing federated matrix factorization methods coordinate local and global updates through local computation strategies, consensus constraints, and representation sharing. Regarding local computation strategies for models with smooth regularization and convex constraints, federated model averaging (FedMAvg) \cite{wang2022federated} and the differentially private federated clustering (DP-FedC) method ~\cite{li2023differentially} progressively reduce local updates of the shared matrix to limit client drift. Alternatively, federated gradient sharing (FedMGS) \cite{wang2022federated} updates the shared matrix at the server using client gradients and exhibits low sensitivity to data heterogeneity. FedMAvg, FedMGS, and DP-FedC provide sublinear convergence bounds for first-order stationarity measures under their respective assumptions.

Some methods introduce consensus constraints between local copies and a shared matrix, which are often handled using Lagrange multipliers or quadratic penalty terms. For nonsmooth regularization, Hytla et al.~\cite{Hytla2026} employ local and server proximal updates within an alternating direction method of multipliers (ADMM) framework. Dual updates and quadratic penalties promote consensus between local copies and the shared matrix. Under the stated assumptions, their analysis establishes a sublinear convergence bound for a first-order stationarity measure. For orthogonality constraint, Huang et al.~\cite{huang2022efficient} proposed federated multi-view learning (FedMVL), using a smooth orthogonality penalty and introducing auxiliary variables with consensus constraints. Their detailed analysis establishes linear convergence of the expected objective gap for the subproblem, without establishing convergence for the full alternating algorithm. Feng et al.~\cite{ijcai2024p439} integrated matrix factorization with $k$-means, incorporating basis orthogonality constraints, an orthogonality penalty on clustering representations, and graph regularization. Its convergence behavior is supported empirically through objective value trajectories. Their alternating algorithm uses an augmented Lagrangian formulation to coordinate local and global clustering representations. These consensus constraints may help mitigate client drift.

Other methods focus on the structure of representation sharing and the efficiency of information exchange between clients and the server. Li et al.~\cite{li2024partially} proposed partially shared federated multi-view learning (PSFedML), combining NMF with graph regularization and partial representation sharing. The server learns a common representation using selection matrices and broadcasts it to clients while preserving information specific to each view. Its convergence analysis establishes convergence of the objective values. Liu et al.~\cite{Liu2026TPAMI} proposed communication-efficient federated multi-view clustering (CeFMC), which incorporates linear kernel information and approximates local representations using pseudo-labels and centroids. Exchanging these quantities between clients and the server reduces communication overhead. Its analysis establishes the objective value convergence through monotonic improvement and boundedness. 

    However, Problem~\eqref{FMF} combines nonconvex constraints on local matrices with nonsmooth regularization on the shared matrix. Existing methods offer strategies for handling orthogonality \cite{huang2022efficient,ijcai2024p439} or nonsmooth regularization \cite{Hytla2026}, but do not explicitly consider the combination studied. This combination introduces several challenges for federated algorithm design and convergence analysis. Firstly, the proximal mapping associated with the shared matrix generally does not commute with averaging.  Therefore, the server aggregation scheme must account for local regularization to maintain consistency with the global objective. Secondly, multiple local updates can cause client drift, particularly under data heterogeneity. Updates to local matrices further affect the update directions of local copies, making it harder to align local updates with the global objective. Thirdly, empirical stabilization \cite{ijcai2024p439} or objective value convergence \cite{Liu2026TPAMI} alone does not establish stationarity. Existing stationarity guarantees often rely on convex constraints or smooth regularization \cite{wang2022federated,li2023differentially}. Establishing convergence guarantees for Problem~\eqref{FMF} requires jointly analyzing nonconvex local constraints, nonsmooth global regularization, and client drift. These challenges raise a natural question, how to develop a federated algorithm for Problem~\eqref{FMF} that handles nonsmooth regularization, controls client drift, and admits convergence guarantees?
	
	
\subsection{Our Contributions}\label{subsec13}		
Compared with the existing works, the main contributions of this paper are summarized as follows.
	\begin{enumerate}
    \item[(i)] \textbf{(Decoupled proximal updates)} 
    For federated sparse ONMF, we develop a federated alternating minimization algorithm that introduces variable splitting to decouple proximal operator evaluation from communication, addressing the noncommutativity of the proximal mapping and averaging.
	
	 \item[(ii)] \textbf{(Drift correction)}
    We incorporate a correction term into the gradients with respect to local copies to align local update directions with the global objective. This correction mitigates client drift induced by multiple local updates, particularly under data heterogeneity. 

    \item[(iii)] \textbf{(Convergence guarantee)}
    Under additional mild conditions, we establish a sublinear convergence rate for the averaged stationarity residual (see Theorem~\ref{Theorem3}). Numerical experiments on a synthetic dataset also support the convergence results, while those on real-world multi-view datasets demonstrate competitive clustering performance.
\end{enumerate}

	\subsection{Organization}\label{subsec14}
The structure of this paper is as follows.
Section \ref{sec2} introduces some notations and preliminary works.
Section \ref{sec3} establishes the main algorithmic frameworks and conducts convergence analysis.
Section \ref{sec4} validates the superiority of the proposed method.
Section \ref{sec5} concludes this paper.


\section{Notations and Preliminaries}\label{sec2}
\subsection{Notations}\label{subsec15}
	In this paper, denote matrices by capital letters, e.g., $A$, vectors by boldface lowercase letters, i.e.,~$\mathbf{a}$, and scalars by lowercase letters, i.e.,~$a$. For a matrix $A\in\mathbb{R}^{d\times n}$, denote its $i$th row, $j$th column, and $(i,j)$th entry by $A_{i\cdot}$, $A_{\cdot j}$, and $A_{ij}$, respectively. The transpose of $A$ is denoted by $A^{\top}$. The Frobenius norm of $A$ is defined as $\|A\|_{F} =  \sqrt{\sum_{i=1}^{d}\sum_{j=1}^{n} A_{ij}^2}$. The $\ell_{1}$-norm of $A$ is defined as $\|A\|_{1}=\sum\limits_{i=1}^{d}\sum\limits_{j=1}^{n}|A_{ij}|$, where $|\cdot|$ represents the absolute value. As the extension of $\ell_{1}$-norm, the $\ell_{2,1}$-norm is given by $ \|A\|_{2,1} = \sum_{i=1}^d \|A_{i\cdot}\|_2$. In addition, $\text{diag}(A)$ represents a vector consisting of diagonal elements of the matrix $A$. For a point $X$ and a nonempty set $\mathcal S$ in a matrix space, let $\text{dist} (X,\mathcal S)=\inf_{Y\in\mathcal S}\|X-Y\|_F$ denotes the distance from $X$ to $\mathcal S$. For a function $h(\cdot)$, let $\text{dom}~h=\{W:h(W)<+\infty\}$ denote its effective domain. A further notation will be introduced wherever it appears. 
    
\subsection{Preliminaries}\label{subsec22}

Define
\begin{equation}\label{eq:regularizer}
	h(V) = \mu\|V\|_{2,1} + \delta_{\mathbb{R}_+^{r\times n}}(V),
\end{equation}
where $\delta_{\mathcal C}(V)=0$ if $V\in\mathcal C$ and $+\infty$ otherwise. For a parameter $\eta>0$, the proximal mapping of $h$ is defined by 
\begin{equation}\label{proxdefinition}
	\mathcal P_{\eta\mu}(Z) =\text{prox}_{\eta\mu h}(Z)
	=\arg\min_{V\geq0} \{
	\mu\|V\|_{2,1}+\frac{1}{2\eta}\|V-Z\|_F^2 \}.
\end{equation}
  Let $A=[Z]_+$, where $[Z]_+$ denotes the entrywise nonnegative part of $Z$. The proximal mapping admits the following closed-form expression for each row
$i=1,\cdots,d$,
\begin{equation*}\label{eq:pre-prox-formula}
	(\mathcal P_{\eta\mu}(Z))_{i \cdot}
	=\begin{cases} 
	 (1-\frac{\eta\mu}{\|A_{i \cdot}\|_2} )A_{i \cdot},
		&\|A_{i \cdot}\|_2>\eta\mu,\\ 
		0,&\|A_{i \cdot}\|_2\leq\eta\mu.
	\end{cases} 
\end{equation*}

As the proximal mapping of a proper closed convex function, $P_{\eta\mu}$ is nonexpansive, for all $Z,W\in\mathbb{R}^{r\times n}$ it has
\begin{equation*}\label{eq:prox-nonexpansiveness}
 \| \mathcal P_{\eta\mu}(Z) -\mathcal P_{\eta\mu}(W) \|_F \leq \|Z-W\|_F.
\end{equation*}
However, interchanging the proximal mapping and averaging yields different results. In particular,
\begin{equation*}\label{eq:prox-averaging}
	\mathcal P_{\eta\mu}
	(\frac{1}{m}\sum_{p=1}^{m}Z_p )
	\neq \frac{1}{m}\sum_{p=1}^{m} \mathcal P_{\eta\mu}(Z_p).
\end{equation*}
\begin{example}
	In the scalar case with $\eta\mu=1$, the mapping reduces to $Q(z)=\max\{z-1,0\}$. Taking $z_1=0$ and $z_2=2$ gives \begin{equation*} Q (\frac{z_1+z_2}{2} )=0,~
\frac{Q(z_1)+Q(z_2)}{2}=\frac{1}{2}.\end{equation*}
\end{example}
 

 The following lemma provides a comparison inequality for a proximal update with an inexact gradient direction and will be used in the subsequent convergence analysis.
\begin{lemma}\label{lemma0}
	Suppose that $\varphi(\cdot)$ is continuously differentiable with an $L$-Lipschitz gradient on a convex region containing $Z,Y,W$, and let $h(\cdot)$ be proper, closed, and convex. If $Y=P_{\eta\mu}(Z-\eta D)$ for $\eta>0$ and $W\in\text{dom}~h(\cdot)$, then
	\begin{equation*}\label{ preinexactproximal}
		\begin{aligned}
			\varphi(Y)+h(Y) &\leq \varphi(W)+h(W) 	+\langle\nabla\varphi(Z)-D,Y-W\rangle\\
			&-\frac{1}{\eta}\langle Y-Z,Y-W\rangle +\frac{L}{2}\|Y-Z\|_F^2 +\frac{L}{2}\|W-Z\|_F^2.
		\end{aligned}
	\end{equation*}
\end{lemma}

\begin{proof}
	Following the proximal optimality condition and the subgradient inequality, it derives \begin{equation*}
		\begin{aligned}
        h(Y)-h(W)\leq-\langle D,Y-W\rangle-\frac{1}{\eta}\langle Y-Z,Y-W\rangle.	
        \end{aligned}
	\end{equation*}	The Lipschitz continuity of $\nabla\varphi$ implies
	\begin{equation*}
		\begin{aligned}
        \varphi(Y)-\varphi(W)	\leq\langle\nabla\varphi(Z),Y-W\rangle+\frac{L}{2}\|Y-Z\|_F^2 +\frac{L}{2}\|W-Z\|_F^2.	
        \end{aligned}
	\end{equation*}
	Adding the two inequalities proves the result.
\end{proof}

	\section{A Federated Alternating Minimization Algorithm }\label{sec3}  
	
	Define the local loss functions $f_p(U_p,V)=\|X_p-U_pV\|_F^2$ and the regularization function $g(V)=\|V\|_{2,1}$. Problem~\eqref{FMF} can be rewritten as
	\begin{equation}\label{FMF1}
		\begin{aligned}
			\min_{\{U_p\},V}&\quad F(\{U_p\},V)=\frac1m\sum_{p=1}^{m}f_p(U_p,V)+\mu g(V)\\
			\text{s.t.}&~~~ U_p\geq0,~U_p^\top U_p=I,~V\geq0.
		\end{aligned}
	\end{equation}
In this section, we develop a federated alternating minimization algorithm for Problem~\eqref{FMF1}. 
Following the model averaging principle of FedAvg, each client $p$ maintains a local copy $V_p$ of the shared matrix $V$. The corresponding local problem is
 \begin{equation}\label{FMFlocal}
		\begin{aligned}
			\min_{U_p,V_p}&\quad f_p(U_p,V_p)+\mu g(V_p)\\
			\text{s.t.}&~~~ U_p\geq0,~U_p^\top U_p=I,~V_p\geq0.
		\end{aligned}
	\end{equation}
At each communication round, clients perform multiple local alternating updates of $U_p$ and $V_p$. The server then aggregates the pre-proximal variables uploaded by the clients and broadcasts the global pre-proximal variable for the next round. The local updates and server aggregation are described below.

	\subsection{Update of the local matrix ${U}_{p}$}\label{seclocal31}
        At communication round $s$, client $p$ performs $\tau$ local alternating updates, indexed by $k=0,\cdots,\tau-1$. Given the current iterates $U_p^{s,k}$ and $V_p^{s,k}$, it updates $U_p^{s,k}$ to $U_p^{s,k+1}$ by solving the following proximal subproblem,
		\begin{equation}\label{Um1}
			\begin{aligned}
				\min_{U }&\quad f_p(U,V_{p}^{s,k}) +\frac{\rho_1}{2}\| U  - U_{p}^{s,k}\|_{F}^{2} \\
				\text{s.t.}&\quad U  \ge 0,~U^{\top} U  = I,
			\end{aligned}
		\end{equation}
		where $\rho_1$ is a positive constant and $U_p^{s,0}=U_p^{s-1}$, with $U_p^{s-1}$ denoting the local matrix retained from the preceding round.
		
		Let $l(U)=f_p(U,V_{p}^{s,k})+\frac{\rho_1}{2}\|U-U_p^{s,k}\|_F^2$. Define the oblique manifold
\begin{equation*}\label{eq:oblique-manifold}
	\mathcal M_p=\{U\in\mathbb R^{d_p\times r}: \|U_{.j}\|_2=1,\ j=1,\cdots,r\},
\end{equation*}
and let $\mathcal B_p=\mathcal M_p\cap\mathbb R_+^{d_p\times r}$. Following the exact penalty framework in \cite{jiang2023exact}, it considers the penalized subproblem
		\begin{equation}\label{penalizedU}
			\min_{U\in\mathcal B_p} P_\sigma(U)=l(U)+\sigma(\|U\mathbf a\|_2^2-1),
		\end{equation}
		where $\sigma>0$ is the penalty parameter, and $\mathbf a=\mathbf e/\sqrt r$ with $\mathbf e\in\mathbb R^r$ denoting the all ones vector.
		
		For $U\in\mathcal B_p$, the penalty term is nonnegative and vanishes if and only if $U^\top U=I$. To characterize stationarity for \eqref{penalizedU}, it equips $\mathcal M_p$ with the metric induced by the Euclidean inner product $\langle D_1,D_2\rangle=\operatorname{tr}(D_1^\top D_2)$. The tangent space at $U\in\mathcal M_p$ is
\begin{equation*}\label{eq:oblique-tangent}
	\mathcal T_{\mathcal M_p}(U)=\{D\in\mathbb R^{d_p\times r}\mid (U)_{\cdot j}^{\top}D_{\cdot j}=0,\ j=1,\cdots,r\}.
\end{equation*}

The Riemannian gradient of $P_\sigma$ is the orthogonal projection of its Euclidean gradient onto $\mathcal T_{\mathcal M_p}(U)$. Following \cite[Eq.~(2.5)]{jiang2023exact}, it is given by \begin{equation*}
			\begin{aligned}
				\text{grad}~P_{\sigma}(U)=\nabla P_{\sigma}(U)-U \text{diag}~(U^{\top}\nabla P_{\sigma}(U))
			\end{aligned}
		\end{equation*}
       The nonnegativity constraint is imposed separately. By \cite[Theorem~4.1]{jiang2023exact}, any local minimizer $U^*$ of \eqref{penalizedU} satisfies \begin{equation*}
    \min \{ U^*,\text{grad}~P_\sigma(U^*) \}=0,
\end{equation*}
where the minimum is taken entrywise.

    This necessary condition motivates the stationarity stopping criterion
\begin{equation*}
			\|\min\{U,\text{grad}~P_\sigma(U)\}\|_F\leq\varepsilon_{1},
		\end{equation*}
		where $\varepsilon_{ 1}>0$ is the tolerance. Algorithm~\ref{PEP} uses this criterion in the numerical solution of the penalized subproblem. After $\tau$ local alternating updates, it denotes the resulting local matrix by $U_p^s=U_p^{s,\tau}$.
        
   \begin{algorithm}[h]
			\caption{Optimization algorithm for solving subproblem~\eqref{Um1}}
			\label{PEP}
			\KwIn{ Matrices $U_p^{s,k}$, $V_p^{s,k}$, initial penalty parameter $\sigma_0>0$, penalty factor $\gamma>1$, tolerances $\varepsilon_{1},\epsilon_2>0$, maximum penalty iterations $T_{a}$
			}
			\KwOut{$U_p^{s,k+1}$}
			
			Set $Y^0=U_p^{s,k}$ 
			
			\For{$t=0,1,\cdots,T_{a}-1$}{
				
				Starting from $Y^t$, use a projected gradient method to find $Y^{t+1}\in\mathcal B_p$ satisfying
				\begin{equation*}
					\| \min \{ Y^{t+1}, \text{grad}~P_{\sigma_t}(Y^{t+1})\} \|_F \leq\varepsilon_{1}  
				\end{equation*}
				and
				\begin{equation*}
					P_{\sigma_t}(Y^{t+1})\leq P_{\sigma_t}(U_p^{s,k}) 
				\end{equation*}
				
				\If{
					$ \|Y^{t+1}\mathbf a\|_2^2-1  \leq\epsilon_2$
				}{
					Set $U_p^{s,k+1}=Y^{t+1}$ 
			
				}
				Set $\sigma_{t+1}=\gamma\sigma_t$ 
                
				Set $Y^{t+1}$ as the initial point for the next penalty iteration 
			}
			Set $U_p^{s,k+1}=Y^{T_{a}}$ 
			
			\Return{$U_p^{s,k+1}$}
		\end{algorithm}   
		
		\begin{remark}
		Under the assumptions and penalty parameter conditions of Theorem 3.1 of \cite{jiang2023exact}, a global minimizer of the penalized problem can be used to recover a global minimizer of \eqref{Um1}. In the convergence analysis, $U_p^{s,k+1}$ denotes the recovered global minimizer.
		\end{remark}

	\subsection	{Update of the local copy ${V}_{p}$ }\label{seclocal32}
			At local iteration $k$ of communication round $s$, client $p$ first obtains $U_p^{s,k+1}$ and then updates the local copy from $V_p^{s,k}$ to $V_p^{s,k+1}$. Specifically, the local update solves
		\begin{equation}\label{Vm}
			\begin{aligned}
				\min_{V_{p}}  &\quad f_p(U_{p}^{s,k+1},V_{p})+ \mu g(V_p) \\
				\text{s.t.}  &\quad V_{p} \ge 0.
			\end{aligned}
		\end{equation}
        
        Solving Problem \eqref{Vm} in the federated setting presents two challenges. On the one hand, the proximal mapping associated with the nonsmooth regularizer and the nonnegativity constraint does not commute with averaging, complicating the design of local updates and server aggregation. On the other hand, multiple local updates between communication rounds can cause client drift. The following two subsections address these challenges through variable splitting and gradient correction, respectively.
	\subsubsection{Variable splitting}
		We use variable splitting to separate gradient accumulation from the proximal mapping. For a fixed round $s$, it omits the round index and write the local iterates as $U_p^k$, $\widehat V_p^k$, and $V_p^k$. Here, the pre-proximal variable $\widehat V_p^k$ accumulates smooth gradient updates, and the post-proximal variable $V_p^k$ is obtained by applying the proximal mapping.
        
At local iteration $k$, the gradient with respect to $V_p$ is
		\begin{equation*}\label{local_smooth_gradient}
			\begin{aligned}
				\nabla_{V_{p}} f_p(U_p^{k+1},V_p^{k})=2(U_p^{k+1})^\top(U_p^{k+1}V_p^{k}-X_p).
			\end{aligned}
		\end{equation*}
		The pre-proximal variable and the post-proximal variable are updated as follows
		\begin{equation*}\label{uncorrected_updates}
			\begin{aligned}
			 \widehat V_p^{k+1} &= \widehat V_p^{k}-\alpha\nabla_{V_p} f_p(U_p^{k+1},V_p^{k}),\\
			V_p^{k+1} &= \mathcal P_{(k+1)\alpha\mu}(\widehat V_p^{k+1}),
			\end{aligned}
		\end{equation*}
		where $\alpha>0$ is the local step size, and $\mathcal P_{\eta\mu}$ with $\eta=(k+1)\alpha$ is defined in \eqref{proxdefinition}. Both variables are initialized as $\widehat V_p^0=V_p^0=V^s=\mathcal P_{\tilde\eta\mu}(\overline V^s)$, where $\overline V^s$ is the global pre-proximal variable broadcast by the server at the beginning of round $s$ and $\tilde\eta=\eta_g\alpha\tau$ is the effective global step size, with $\eta_g>0$ denoting the server step size. 
        \begin{remark}
        At local iteration $k$, the proximal parameter is $\eta=(k+1)\alpha$. The pre-proximal variable accumulates gradient updates without incorporating the results of the proximal mappings. Thus, each post-proximal iterate is obtained by applying the proximal mapping to the initial point of the round adjusted by the accumulated gradients. As explained by Zhang et al.~\cite{zhang2026nonconvex}, the proximal step size should match the cumulative step size of these gradient updates, which is $(k+1)\alpha$ after $k+1$ local updates.
        \end{remark}
			This variable splitting strategy separates gradient accumulation from the proximal mapping. The post-proximal variable incorporates regularization and nonnegativity and is used in subsequent local updates. The pre-proximal variable retains the accumulated update directions and is uploaded for server aggregation. Thus, the algorithm avoids interchanging proximal mapping and averaging. It next introduces a correction term into the pre-proximal update to account for client drift.

		\subsubsection{Correction for client drift}
		During local alternating updates, both the local matrices and local copies evolve. Their changes can produce different gradient directions across clients, particularly when multiple local iterations precede server aggregation.

Existing federated methods address client drift through different mechanisms. Li et al. \cite{li2020federated} introduce a quadratic proximal term to limit deviations from the current global model. Karimireddy et al. \cite{karimireddy2020scaffold} use client and server control variates to correct local gradient directions. Acar et al. \cite{acar2021federated} dynamically regularize local objectives to align local solutions with the global objective in the limit. For composite federated optimization, Zhang et al. \cite{zhang2026nonconvex} construct a correction from the difference between the global and local average gradients computed in the preceding round. Their mechanism recovers aggregate gradient information from the broadcast pre-proximal variable without separately communicating correction matrices.

We adopt the gradient correction mechanism proposed by \cite{zhang2026nonconvex}. At the beginning of communication round $s\geq2$, client $p$ computes the correction using the newly broadcast variable $\overline V^s$, the stored shared variable $V^{s-1}$, and the local gradients evaluated during round $s-1$,
\begin{equation}\label{cor}
	C_p^s 	= \frac{V^{s-1}-\overline {V}^s}{\tilde\eta} 	- \frac{1}{\tau}\sum_{k=0}^{\tau-1}
	\nabla_{V_p^{s-1}} f_p(U_p^{s-1,k+1},V_p^{s-1,k}),
\end{equation}
where $C_p^1=0$ and the correction remains fixed throughout round $s$. 

For $k=0,\cdots,\tau-1,$ incorporating \eqref{cor} into the variable splitting scheme yields
\begin{subequations}\label{corrected_updates}
	\begin{align}
		\widehat V_p^{s,k+1} &= \widehat V_p^{s,k} -\alpha (
		\nabla_{V_p} f_p(U_p^{s,k+1},V_p^{s,k})
		+C_p^s ), \label{correctedpreprox}\\
		V_p^{s,k+1} &= 	\mathcal P_{(k+1)\alpha\mu} 	(\widehat V_p^{s,k+1}). 	\label{correctedprox}
	\end{align}
\end{subequations}
After completing these updates, client $p$ uploads $\widehat V_p^{s,\tau}$ to the server.

To clarify the correction mechanism, define the local average gradient as
\begin{equation*}\label{local_average_gradient}
	G_p^s 	= \frac{1}{\tau}\sum_{k=0}^{\tau-1}
	\nabla_{V_p^{s}} f_p(U_p^{s,k+1},V_p^{s,k}).
\end{equation*}
Accumulating \eqref{correctedpreprox} over $\tau$ local updates gives
\begin{equation}\label{accumulated_preprox}
	\widehat V_p^{s,\tau} = V^{s}-\alpha\tau(G_p^{s}+C_p^{s}).
\end{equation}
Substituting this identity into the server update in Section~\ref{secglobal33} and using $\tilde\eta=\eta_g\alpha\tau$ yields
\begin{equation}\label{agggrad}
	\frac{V^{s}-\overline {V}^{s+1}}{\tilde\eta} = \frac{1}{m}\sum_{p=1}^{m} (G_p^{s}+C_p^{s}).
\end{equation}
Since $C_p^1=0$, the zero-sum property $\sum_{p=1}^{m}C_p^s=0$ follows by induction. Indeed, if it holds at round $s$, then \eqref{cor} and \eqref{agggrad} imply
\begin{equation}\label{cordiff}
    C_p^{s+1}=\frac{1}{m}\sum_{q=1}^{m}G_q^s-G_p^s.
\end{equation}
Summing over all clients gives $\sum_{p=1}^{m}C_p^{s+1}=0$, completing the induction. Equation~\eqref{cordiff} shows that the correction for round $s+1$ is the difference between the global and local average gradients computed in round $s$. Each client recovers this difference from the server feedback and its stored local average gradient, without separately communicating correction matrices.

	\subsection{ Update of the shared matrix $V$}\label{secglobal33}
	At the end of communication round $s$, the server receives the local pre-proximal variables $\widehat V_p^{s,\tau}$ from all clients. Using their average and the current shared variable $V^s=\mathcal P_{\tilde\eta\mu}(\overline V^s)$, it computes the updated global pre-proximal variable as
	\begin{equation}\label{barV}
		\overline {V}^{s+1}=V^s+\eta_g(\frac{1}{m}\sum_{p=1}^{m}\widehat V_p^{s,\tau}- {V}^s)),
	\end{equation}
	where $\eta_g>0$ is the server step size. The updated pre-proximal variable $\overline {V}^{s+1}$ is then broadcast to all clients for the next communication round. After completing $T$ communication rounds, the final shared matrix is given by
	\begin{equation*}\label{final_global_output}
		V^{T+1}=\mathcal{P}_{\tilde{\eta}\mu} (\overline {V}^{T+1}).
	\end{equation*}
	In summary, Algorithm~\ref{al1} provides the full implementation details of our algorithm from the perspective of a single client, while Fig.~\ref{alblock} illustrates the local updates and communication between the client and the server.

	\begin{algorithm}[t]
		\caption{Federated alternating minimization algorithm for solving \eqref{FMF1}}\label{al1}
		\KwIn{Matrix $ X_p\in\mathbb R_+^{d_p\times n}$, parameters $\alpha,\eta_g,\mu,\rho_1>0,$ the maximum communication round $T$, the maximum number of iterations $\tau$, with $\tilde{\eta}=\eta_g\alpha\tau$ }
		
		\textbf{Initialize} $\overline V^1\geq 0$, $V^1=\mathcal P_{\tilde\eta\mu}(\overline V^1)$, $C_p^1=0$, and $U_p^{1,0}\geq 0$
		
		\For{$s = 1,2,\dots, T$}{
        
			\textbf{Client $p$}
			
				 Receive $\overline {V}^{s}$ from the server
		 
			Update correction matrix $C_p^{s}$ by \eqref{cor}
			
				Initialize $\widehat V_p^{s,0}=V_p^{s,0}=\mathcal P_{\tilde\eta\mu}(\overline {V}^{s})$ and $U_p^{s,0}=U_p^{s-1,\tau}$
				
				\For{$k=0,1,\cdots,\tau-1$}{
					
					Update local basis matrix $U_p^{s,k+1}$ by Algorithm \ref{PEP} 
                    
					Update pre-proximal matrix $\widehat V_p^{s,k+1}$ by \eqref{correctedpreprox}
                    
					Update post-proximal matrix $V_p^{s,k+1}$ by \eqref{correctedprox}
				}
			
			Upload $\{\widehat V_p^{s,\tau}\}$ to the server
			
			\textbf{Server}

			Updates $\overline {V}^{s+1}$ by \eqref{barV}
			
			Broadcasts $\overline {V}^{s+1}$ to all the clients 
			
		}
		\KwOut{The shared matrix $V^{T+1}=\mathcal{P}_{\tilde{\eta}\mu}(\overline {V}^{T+1})$}
		
	\end{algorithm}

\begin{figure}[htbp]
\centering
\begin{tikzpicture}[
    x=1cm,
    y=1cm,
    >=Stealth,
    flow/.style={
        ->,
        draw=blue,
        line width=1.2pt
    },
    block/.style={
        draw=black,
        dashed,
        line width=1.1pt,
        rounded corners=9pt
    },
    highlight/.style={
        draw=blue,
        line width=1.1pt,
        rounded corners=3pt,
        inner sep=4pt
    },
    labeltext/.style={
        font=\sffamily
    }
]

\draw[block] (1.7,0) rectangle (5.5,3.4);

\draw[block] (7.15,0) rectangle (9.45,3.4);

\draw[flow] (0.15,1.5) -- (1.7,1.5);
\draw[flow] (5.5,1.5) -- (7.15,1.5);
\draw[flow] (9.45,1.5) -- (11.3,1.5);

\node[labeltext,anchor=west] at (0.1,2.6) {Receive};
\node[anchor=west] at (0.1,2.1)
    {$\overline V^s$};
    
\node at (3.08,2.85)
    {$U_p^{s,1}$};
\node at (4.00,2.85)
    {$\cdots$};
\node at (4.83,2.85)
    {$U_p^{s,\tau}$};
    
\node at (2.15,1.82)
    {$C_p^s$};
\node at (3.08,1.82)
    {$\widehat V_p^{s,1}$};
\node at (4.00,1.82)
    {$\cdots$};
\node[highlight] at (4.83,1.82)
    {$\widehat V_p^{s,\tau}$};

\node at (3.08,0.58)
    {$V_p^{s,1}$};
\node at (4.00,0.58)
    {$\cdots$};
\node at (4.83,0.58)
    {$V_p^{s,\tau}$};

\node[labeltext] at (6.25,2.65) {Send};
\node at (6.32,2.05)
    {$\widehat V_p^{s,\tau}$};

\node[highlight] at (8.3,1.82)
    {$\overline {V}^{s+1}$};
\node at (8.3,0.58)
    {$\mathcal{P}_{\tilde{\eta}\mu} (\overline {V}^{s+1})$};

\node[labeltext,anchor=west] at (9.65,2.6)
    {Broadcast};
\node[anchor=west] at (9.75,2.05)
    {$\overline {V}^{s+1}$};

\node[labeltext] at (3.6,-0.5) {Client $p$};
\node[labeltext] at (8.3,-0.5) {Server};

\end{tikzpicture}
\caption{Block diagram of Algorithm~\ref{al1}.}
\label{alblock}
\end{figure}
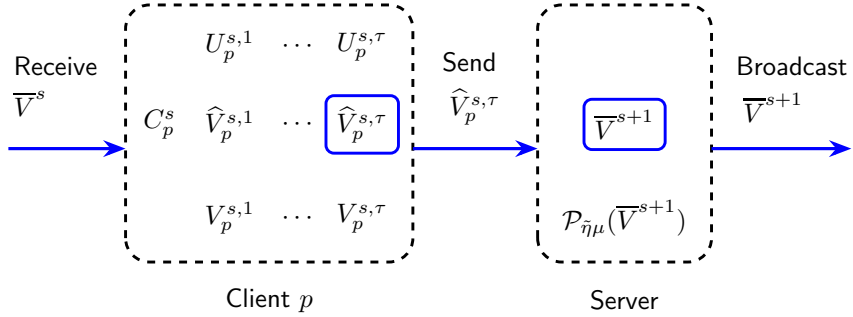
    
	\subsection{Convergence Analysis}\label{subsec32}
	
	 It establishes the convergence results under the following assumptions.
	\begin{assumption}\label{assumption1}
	    There exist constants $L_{1,p}>0$, $p=1,\ldots,m$, such that, for any feasible $U_p$ and any $V_1,V_2$ generated by the algorithm, \begin{equation*}
		\begin{aligned} 
        \|\nabla_{U_p}f_p(U_p,V_1)-\nabla_{U_p}f_p(U_p,V_2)\|_F\leq L_{1,p}\|V_1-V_2\|_F.
        \end{aligned}
	\end{equation*}
	\end{assumption}
\begin{assumption}\label{assumption2}
	  The proximal basis subproblem~\eqref{Um1} and its penalty formulation satisfy the conditions of Theorem~3.1 in \cite{jiang2023exact} at every local update.
	\end{assumption}
    Under Assumption \ref{assumption2}, Theorem~3.1 and Lemma~2.3 of \cite{jiang2023exact} imply that $U_p^{s,k+1}$ is a global minimizer of the original proximal basis subproblem. Consequently, the basis update satisfies both the sufficient descent inequality and the first-order stationarity condition.
    \begin{assumption}\label{assumption3}
    There exists a constant $c_h>0$ such that $\|(X-Y)/\tilde\eta\|_F\leq c_h$ for every proximal evaluation $Y=\mathcal P_{\tilde\eta\mu}(X)$ generated by Algorithm~\ref{al1}.
	\end{assumption}
	
	The following Lipschitz bounds follow directly from the structure of the objective function. For the function $f_p(U_p,V)=\|X_p-U_pV\|_F^2$, the orthogonality constraint $U_p^\top U_p=I$ implies $ \nabla_V f_p(U_p,V)=2V-2U_p^\top X_p.$ For any feasible $U_p$ and any $V_1,V_2$, it holds
	\begin{equation*}
		\|\nabla_V f_p(U_p,V_1)-\nabla_V f_p(U_p,V_2)\|_F
		\leq 2\|V_1-V_2\|_F.
	\end{equation*}
	Moreover, for any feasible $U_{p,1},U_{p,2}$ and any $V$,
	\begin{equation*}
		\begin{aligned}
			&\quad\|\nabla_V f_p(U_{p,1},V) 	-\nabla_V f_p(U_{p,2},V)\|_F \\
            &= 2\|(U_{p,1}-U_{p,2})^\top X_p\|_F \leq 2\|X_p\|_F\|U_{p,1}-U_{p,2}\|_F.
		\end{aligned}
	\end{equation*}
	Thus, on the feasible set, $\nabla_V f_p$ is Lipschitz continuous in $V$ and $U_p$ with constants $L_2=2$ and $L_{3,p}=2\|X_p\|_F$, respectively.

    For the analysis, it defines the virtual iterate $\widetilde{V}^{s+1}$, 
	\begin{equation}\label{virtual}
		\begin{aligned} 
			\widetilde{V}^{s+1} = \mathcal{P}_{\tilde{\eta}\mu} ( V^{s} - \tilde{\eta} \nabla_{V_p} f(\{ U_p^{s,0}\},V^s) ),
		\end{aligned}
	\end{equation}
	where $f(\{U_p\},V) =\frac{1}{m}\sum_{p=1}^{m}f_p(U_p,V)$. This virtual iterate is used only in the analysis and does not require additional computation or communication. 
    
		It first establishes a bound on the drift error between the local copy iterates $V_p^{s,k}$ and their common initial point $V^s$.
	\begin{lemma}\label{LeA.12}
		If the parameter $\tilde\eta$ satisfies $0<\tilde\eta\leq\frac{\eta_g}{\sqrt{20}L_2},$ then, for every $s\geq1$, it holds that
		\begin{equation*}\label{local_drift_bound}
			\begin{aligned}
				\sum_{k=0}^{\tau-1}\sum_{p=1}^{m}\|V_p^{s,k}-V^s\|_F^2	&\le 6\tau\|\Lambda^s-\overline{\Lambda}^s\|_F^2+24m\tau^3\alpha^2c_h^2\\
				&+6m\tau^3\alpha^2\|\widetilde{\nabla}h(\widetilde V^{s+1})
				+\nabla_V f(\{U_p^{s,0}\},V^s)\|_F^2\\
				&+12\alpha^2\tau^4\sum_{p=1}^{m}\sum_{k=0}^{\tau-1}L_{3,p}^2\|\Delta U_p^{s,k}\|_F^2,
			\end{aligned}
		\end{equation*}
		where $\Lambda_{a}^s=\frac1m\sum_{p=1}^{m}\Lambda_p^s,~\Lambda^s=\text{col}~\{\Lambda_1^s,\cdots,\Lambda_m^s\},~\overline{\Lambda}^s=\text{col}~\{\Lambda_{a}^s,\cdots,
		\Lambda_{a}^s\},~\widetilde{\nabla}h(\widetilde V^{s+1})\in\partial h(\widetilde V^{s+1}),~ \Delta U_p^{s,k} =U_p^{s,k+1}-U_p^{s,k},$ and
       $ \Lambda_p^s =\alpha\tau ( \nabla_V f_p(U_p^{s,0},V^s) +\frac{1}{m}\sum_{p=1}^{m}G_p^{s-1} -G_p^{s-1} ).$

		
	\end{lemma}
    With the bound in Lemma~\ref{LeA.12} at hand, it next establishes a recursion for $\|\Lambda^s-\overline{\Lambda}^s\|_F^2$.
	\begin{lemma}\label{LeA.21}
		For every $s=1,\cdots,T$, it holds that
		\begin{equation*}\label{Lambda_recursion_compact}
			\begin{aligned}
				&\frac{1}{m}\|\Lambda^{s+1}-\overline{\Lambda}^{s+1}\|_F^2-	2\alpha^2\tau^2L_2^2	\|V^{s+1}-V^s\|_F^2\\
				&\leq\frac{4\alpha^2\tau L_2^2}{m}\sum_{p=1}^{m}\sum_{k=0}^{\tau-1}\|V_p^{s,k}-V^s\|_F^2+\frac{2\alpha^2\tau^2(\tau-1)}{m}\sum_{p=1}^{m}\sum_{k=0}^{\tau-1}L_{3,p}^2\|\Delta U_p^{s,k}\|_F^2.
			\end{aligned}
		\end{equation*}
	\end{lemma}
It next bounds the change in the objective function as the local matrices are updated from $\{U_p^{s,0}\}$ to $\{U_p^{s+1,0}\}$, with the shared matrix fixed at $V^{s+1}$.
	\begin{lemma}\label{LeA.3}
		Under the update rule $ U_p^{s+1,0}=U_p^{s,\tau},$ for communication round $s$, it holds that
		\begin{equation}\label{roundwise_U_descent_extended}
			\begin{aligned}
				&F(\{U_p^{s+1,0}\},V^{s+1})-F(\{U_p^{s,0}\},V^{s+1})\\
				&\leq -\frac{\rho_1}{4mL_{3,a}^2} \sum_{p=1}^{m}\sum_{k=0}^{\tau-1}L_{3,p}^2\|\Delta U_p^{s,k}\|_F^2 + \frac{1}{m\rho_1} \sum_{p=1}^{m}\sum_{k=0}^{\tau-1} L_{1,p}^2 \|V_p^{s,k}-V^{s+1}\|_F^2,
			\end{aligned}
		\end{equation}
		where $	L_{3,a}=\max_{p}L_{3,p}.$
	\end{lemma}

	 Lemma~\ref{LeA.3} bounds the objective change due to the local matix updates. The next lemma bounds the stationarity residual of the local matrices, which measures how closely these iterates satisfy the first-order stationarity condition.
	\begin{lemma}\label{LeA.4}
		Suppose that Assumption \ref{assumption1}-\ref{assumption3} hold. Then, for communication round $s$, local iteration $k$, and client $p$, it has
		\begin{equation*}
			\begin{aligned}
				r_p^{s,k+1} \leq 
				2L_{1,p}^{2} \|V_p^{s,k}-V^s\|_F^2 + 2\rho_1^2 \|\Delta U_p^{s,k}\|_F^2,
			\end{aligned}
		\end{equation*}
		where
		\begin{equation*}
        \begin{aligned}
			r_p^{s,k+1} &= \text{dist}^{2}~( 0,  \nabla_{U_p}f_p(U_p^{s,k+1},V^s) + N_{\mathcal O_p}(U_p^{s,k+1}) ),\\
            \mathcal O_p &= \{ U\in\mathbb R^{d_p\times r}: U\geq 0,\ U^\top U=I \}
		\end{aligned}
        \end{equation*} is the feasible set for the local matrix, and $N_{\mathcal O_p}(U)$ denotes the limiting normal cone to $\mathcal O_p$ at the matrix $U$.
	\end{lemma}
	
	The proofs of the above lemmas are provided in Appendices~\ref{secA1}, \ref{secA2}, \ref{secA3}, and~\ref{secA4}, respectively.
	 
		It next combines the drift estimates and the local matrix descent bound with the proximal comparison inequality in Lemma~\ref{lemma0}. Before stating the convergence theorem, it introduces the auxiliary function 
\begin{equation*}\label{finalauxiliaryfunction}
    \Phi^s = F(\{U_p^{s,0}\},V^s)-F^\star +\frac{1}{m\tilde\eta} \|\Lambda^s-\overline{\Lambda}^s\|_F^2,
\end{equation*}
where $F^\star$ denotes the optimal value of problem~\eqref{FMF1}.
The auxiliary function $\Phi^s$ combines the objective gap with a measure of the gradient direction discrepancy. For convenience, define
\begin{equation}\label{defCsAs}
    C_s=\frac{1}{m}\|\Lambda^s-\overline{\Lambda}^s\|_F^2, ~
    A_s=\frac{1}{m}\sum_{p=1}^{m}\sum_{k=0}^{\tau-1}
    L_{3,p}^2\|\Delta U_p^{s,k}\|_F^2,
\end{equation}
and
\begin{equation}\label{defRVS}
    R_V^s = \|\widetilde{\nabla}h(\widetilde V^{s+1}) +\nabla_V f(\{U_p^{s,0}\},V^s)\|_F^2.
\end{equation}
The following theorem establishes a finite-round bound on the average stationarity residual for $V$.
	\begin{theorem}\label{Theorem}
	Under Assumptions~\ref{assumption1}-\ref{assumption3}, let $L_{1,a}=\max_{p}L_{1,p}$. If the step sizes and the proximal parameter satisfy
		\begin{equation}\label{final_parameter_conditions}
			\begin{aligned}
				& 0<\tilde{\eta}\leq\frac{1}{10L_2},~\eta_g\geq\frac{3}{2},\\
				&\rho_1> \max\{ 72\tau\tilde{\eta}L_{1,a}^2, 8L_{3,a}^2( \tilde{\eta} ( \frac{\tau+1}{2} + 12L_2^2\alpha^2\tau^3 ) \\
				&+ \frac{ 2\alpha^2\tau^2(\tau-1) + 48\alpha^4\tau^5L_2^2 }{\tilde{\eta}}), 8\sqrt{3} L_{3,a}L_{1,a}\alpha\tau^2 \}.
			\end{aligned}
		\end{equation}
		Then, there exist constants $c_1,c_2,c_3>0$ such that, for every $T\geq1$,
		\begin{equation*}\label{eq:joint-descent-bound}
			\frac{1}{T}\sum_{s=1}^{T} (c_1R_V^s +c_2C_s +c_3 A_s ) \leq \frac{\Phi^1}{T}
			+c_4 c_h^2.
		\end{equation*}
		Consequently,
		\begin{equation*}\label{eq:V-residual-bound}
			\frac{1}{T}\sum_{s=1}^{T}R_V^s \leq \frac{\Phi^1}{c_1T}+ \frac{c_4}{c_1}c_h^2,
		\end{equation*}
		where $c_1=(\tilde{\eta} (\frac12-L_2\tilde{\eta} -6L_2^2\tau^2\alpha^2 ) - \frac{24\alpha^4\tau^4L_2^2}{\tilde{\eta}} - \frac{12\tau^3\alpha^2L_{1,a}^2}{\rho_1}),~ c_4 = 24\tilde{\eta}L_2^2\tau^2\alpha^2 + \frac{96\alpha^4\tau^4L_2^2}{\tilde{\eta}} + \frac{48\tau^3\alpha^2L_{1,a}^2}{\rho_1},$ and under \eqref{final_parameter_conditions}, it holds that $c_1>0$.
	\end{theorem}
    
	\begin{proof}
    The proof proceeds in four steps.

    \textbf{Step 1: It reformulates the global update for $V^{s+1}$.}
    
		Algorithm~\ref{al1} gives $V^s=\mathcal P_{\tilde\eta\mu}(\overline V^s)$. The definition of $f(\{U_p\},V)$ yields
		\begin{equation*}
			\nabla_V f(\{  U^{s,0}\},V^s)
			=	\frac{1}{m}\sum_{p=1}^{m}	\nabla_V f_p(U_p^{s,0},V^s).
		\end{equation*}
		and introduce the averaged local gradient direction
        \begin{equation*}\label{aggregated_direction}
G^s=\frac{1}{m}\sum_{p=1}^{m}G_p^s.
\end{equation*}
		Summing the local updates in Algorithm~\ref{al1} and using
$\widehat V_p^{s,0}=V^s$ yields
		\begin{equation}\label{cumulpreproxupdate}
			\begin{aligned}
				\widehat V_p^{s,\tau}=V^s-\alpha\sum_{k=0}^{\tau-1}	(\nabla_V f_p	(U_p^{s,k+1},V_p^{s,k})	+C_p^s).
			\end{aligned}
		\end{equation}
		Substituting \eqref{cumulpreproxupdate} into the global update gives
		\begin{equation*}
			\begin{aligned}
				\overline {V}^{s+1}&=	V^s-\frac{\eta_g\alpha}{m}\sum_{p=1}^{m}\sum_{k=0}^{\tau-1}
				(\nabla_V f_p(U_p^{s,k+1},V_p^{s,k})+C_p^s)	\\
				&=	V^s	-\eta_g\alpha\tau
				(\frac{1}{m\tau}\sum_{p=1}^{m}\sum_{k=0}^{\tau-1}\nabla_V f_p(	U_p^{s,k+1},V_p^{s,k})+\frac{1}{m}\sum_{p=1}^{m}C_p^s).
			\end{aligned}
		\end{equation*}
		Since the client drift correction terms satisfy $\frac{1}{m}\sum_{p=1}^{m}C_p^s	=0,$
		and $\tilde{\eta}=\eta_g\alpha\tau$, the global pre-proximal variable can be equivalently written as $\overline {V}^{s+1}= V^s-\tilde{\eta}G^s.$ Consequently, by applying the global proximal operator, the global post-proximal iterate satisfies
		\begin{equation*}\label{global_proximal_gradient_form}
			V^{s+1}	=\mathcal P_{\tilde{\eta}\mu}(	V^s-\tilde{\eta}G^s).
		\end{equation*}
        
	\textbf{Step 2: It bounds the change in $F(\{U_p^{s,0}\},V^s)$ over one communication round.}
    
		Applying Lemma~\ref{lemma0} to the virtual iterate the virtual iterate $\widetilde{V}^{s+1}$, using $Y=\widetilde{V}^{s+1},Z=W=V^s,$ and $D=\nabla_V f(\{ U^{s,0}\},V^s)$ yields
		\begin{equation*}
			\begin{aligned}
				F(\{ U^{s,0}\},\widetilde V^{s+1}) &\leq F(\{ U^{s,0}\},{V}^{s}) +(\frac{L_2}{2}- \frac{1}{\tilde{\eta}})\| \widetilde{V}^{s+1} - {V}^{s} \|_{F}^2.
			\end{aligned}
		\end{equation*}
		For the actual iterate $V^{s+1}$, using $Y={V}^{s+1},X={V}^{s} ,W=\widetilde{V}^{s+1}$ and $D=G^{s}$ in Lemma~\ref{lemma0} yields
		\begin{equation*}
			\begin{aligned}
				&\quad F(\{ U^{s,0}\},V^{s+1})\\
                &\leq F(\{ U^{s,0}\},\widetilde{V}^{s+1}) - \frac{1}{\tilde{\eta}} \langle	V^{s+1}-V^s, V^{s+1}-\widetilde V^{s+1} \rangle + \frac{L_2}{2} \| {V}^{s+1} - {V}^{s} \|_{F}^2\\
				&+ \langle \nabla_V f(\{ U^{s,0}\},{V}^{s}) - G^{s}, {V}^{s+1} - \widetilde{V}^{s+1} \rangle   + \frac{L_2}{2} \| \widetilde{V}^{s+1} -  {V}^{s}\|_{F}^2\\
				&\leq F(\{U^{s,0}\},\widetilde V^{s+1}) + \langle \nabla_V f(\{U^{s,0}\},V^s)-G^s,
				V^{s+1}-\widetilde V^{s+1} \rangle \\
				& + (\frac{L_2}{2} - \frac{1}{2\tilde{\eta}}) \|V^{s+1}-V^s\|_F^2 + ( \frac{L_2}{2} + \frac{1}{2\tilde{\eta}} ) \|\widetilde V^{s+1}-V^s\|_F^2 - \frac{1}{2\tilde{\eta}} \|V^{s+1}-\widetilde V^{s+1}\|_F^2.
			\end{aligned}
		\end{equation*}
		
		To eliminate the virtual function value $F(\{ U^{s,0}\},\widetilde{V}^{s+1})$, the above inequalities can be added to obtain
		\begin{equation}\label{66}
			\begin{aligned}
				&\quad F(\{ U^{s,0}\},V^{s+1})\\
				&\leq F(\{ U^{s,0}\},{V}^{s})+ (\frac{L_2}{2}- \frac{1}{2\tilde{\eta}})\| V^{s+1}-V^s \|_{F}^2  \\
                &+ \underbrace{\langle \nabla_V f(\{ U^{s,0}\},V^s)-G^s,
					V^{s+1}-\widetilde V^{s+1} \rangle}_{(\text{I})}\\
				&+ (L_2-\frac{1}{2\tilde{\eta}}) \| \widetilde{V}^{s+1} - V^s \|_{F}^2- \underbrace{\frac{1}{2\tilde{\eta}}\| {V}^{s+1} - \widetilde{V}^{s+1} \|_{F}^2}_{(\text{II})}.
			\end{aligned}
		\end{equation}
		
		By Young's inequality, $(\text{I})-(\text{II}) \leq\|\nabla_V f( \{U^{s,0}\},V^s)-G^s\|_{F}^{2}.$
		Using the definitions of $\nabla_V f(\{U_p^{s,0}\} ,V^s)$ and $G^s$, it has
		\begin{equation*}
			\begin{aligned}
				\nabla_V f( \{U^{s,0}\},V^s)-G^s  = \frac{1}{m\tau} \sum_{p=1}^{m}
				\sum_{k=0}^{\tau-1}( \nabla_V f_p(U_p^{s,0},V^s) - \nabla_V f_p(U_p^{s,k+1},V_p^{s,k})).
			\end{aligned}
		\end{equation*}
		Hence, by Jensen's inequality,
		\begin{equation*}
			\begin{aligned}
				\|\nabla_V f( \{U^{s,0}\},V^s)-G^s\|_F^2 \leq \frac{1}{m\tau} \sum_{p=1}^{m} \sum_{k=0}^{\tau-1} \| \nabla_V f_p(U_p^{s,0},V^s) - \nabla_V f_p(U_p^{s,k+1},V_p^{s,k}) \|_F^2 .
			\end{aligned}
		\end{equation*}
		Moreover, by the Lipschitz continuity properties established above,
		\begin{equation*}
			\begin{aligned}
				& \| \nabla_V f_p(U_p^{s,0},V^s) - \nabla_V f_p(U_p^{s,k+1},V_p^{s,k}) \|_F^2 \\
                &\leq 2L_{3,p}^2 \|U_p^{s,k+1}-U_p^{s,0}\|_F^2 + 2L_2^2 \|V_p^{s,k}-V^s\|_F^2 .
			\end{aligned}
		\end{equation*}
		Combining the above inequalities gives
		\begin{equation*}
			\begin{aligned}
				(\text{I})-(\text{II})
				&\leq \frac{\tilde\eta L_2^2}{m\tau} \sum_{p=1}^{m} \sum_{k=0}^{\tau-1} \|V_p^{s,k}-V^s\|_F^2  + \frac{\tilde\eta}{m\tau} \sum_{p=1}^{m} \sum_{k=0}^{\tau-1} L_{3,p}^2 \|U_p^{s,k+1}-U_p^{s,0}\|_F^2 .
			\end{aligned}
		\end{equation*}
		
		Substituting this equation into \eqref{66}, it has 
		\begin{equation}\label{proof_V_descent_start}
			\begin{aligned}
				F(\{ U^{s,0}\},V^{s+1})  &\leq F(\{ U^{s,0}\},{V}^{s})+ 
				(\frac{L_2}{2}-	\frac{1}{2\tilde{\eta}})\|V^{s+1}-V^s\|_F^2\\
                &+\frac{\tilde{\eta}L_2^2}{m\tau}\sum_{p=1}^{m}\sum_{k=0}^{\tau-1}\|V_p^{s,k}-V^s\|_F^2+ (L_2-\frac{1}{2\tilde{\eta}})\|\widetilde V^{s+1}-V^s\|_F^2 \\
				&+\frac{\tilde{\eta}}{m\tau}\sum_{p=1}^{m}\sum_{k=0}^{\tau-1}
				L_{3,p}^2\|U_p^{s,k+1}-U_p^{s,0}\|_F^2 .
			\end{aligned}
		\end{equation}
		By Lemma~\ref{LeA.3}, the change caused by the local $U_p$ updates between two communication rounds is bounded by 
		\begin{equation}\label{proof_use_lemma3}
			\begin{aligned}
				&\quad F(\{U_p^{s+1,0}\},V^{s+1}) - F(\{U_p^{s,0}\},V^{s+1}) \\
				&\leq -\frac{\rho_1}{4mL_{3,a}^2} \sum_{p=1}^{m}\sum_{k=0}^{\tau-1} L_{3,p}^2 	\|\Delta U_p^{s,k}\|_F^2 + \frac{1}{m\rho_1} \sum_{p=1}^{m}\sum_{k=0}^{\tau-1} L_{1,p}^2 \|V_p^{s,k}-V^{s+1}\|_F^2.
			\end{aligned}
		\end{equation}
		
		The last term in \eqref{proof_use_lemma3} can be controlled as follows. Since $ V_p^{s,k}-V^{s+1} = (V_p^{s,k}-V^s)+(V^s-V^{s+1}),$ it has 
		\begin{equation}\label{proof_cross_error_1} 
			\begin{aligned}
				&\quad\frac{1}{m\rho_1} \sum_{p=1}^{m}\sum_{k=0}^{\tau-1} L_{1,p}^2 \|V_p^{s,k}-V^{s+1}\|_F^2\\
				&\leq \frac{2L_{1,a}^2}{m\rho_1} \sum_{p=1}^{m}\sum_{k=0}^{\tau-1} \|V_p^{s,k}-V^s\|_F^2  + \frac{2\tau L_{1,a}^2}{\rho_1} \|V^{s+1}-V^s\|_F^2 .
			\end{aligned} 
		\end{equation}
		
		Applying Lemma~\ref{LeA.12} to the first term on the right hand side of \eqref{proof_cross_error_1} gives 
		\begin{equation}\label{proof_cross_error_2} 
			\begin{aligned}
				&\quad\frac{2L_{1,a}^2}{m\rho_1} \sum_{p=1}^{m}\sum_{k=0}^{\tau-1} \|V_p^{s,k}-V^s\|_F^2
				\\
				&\leq \frac{12\tau L_{1,a}^2}{m\rho_1} \|\Lambda^s-\overline{\Lambda}^s\|_F^2 
				+ \frac{12\tau^3\alpha^2L_{1,a}^2}{\rho_1} \| \widetilde{\nabla}h(\widetilde V^{s+1}) + \nabla_V f(\{U_p^{s,0}\},V^s) \|_F^2 \\
				&+ \frac{48\tau^3\alpha^2L_{1,a}^2}{\rho_1} c_h^2 + \frac{24\alpha^2\tau^4L_{1,a}^2}{m\rho_1} \sum_{p=1}^{m}\sum_{k=0}^{\tau-1} L_{3,p}^2 \|\Delta U_p^{s,k}\|_F^2 .
			\end{aligned} 
		\end{equation}
        
		\textbf{Step 3: It establishes a recursion for the auxiliary function $\Phi^s$.}
        
		Next, multiplying Lemma~\ref{LeA.21} by $1/\tilde{\eta}$ and using Lemma~\ref{LeA.12}, it obtains
		\begin{equation}\label{proof_Lambda_scaled}
			\begin{aligned}
				\frac{1}{m\tilde{\eta}} \|\Lambda^{s+1}-\overline{\Lambda}^{s+1}\|_F^2 	&\leq \frac{2\alpha^2\tau^2L_2^2}{\tilde{\eta}} \|V^{s+1}-V^s\|_F^2 + \frac{24\alpha^2\tau^2L_2^2}{m\tilde{\eta}}
				\|\Lambda^s-\overline{\Lambda}^s\|_F^2 \\
				&+ \frac{96\alpha^4\tau^4L_2^2}{\tilde{\eta}}c_h^2 + \frac{24\alpha^4\tau^4L_2^2}{\tilde{\eta}}\|
				\widetilde{\nabla}h(\widetilde V^{s+1}) + 	\nabla_V f(\{U_p^{s,0}\},V^s) \|_F^2 \\
				&\quad+ \frac{ 2\alpha^2\tau^2(\tau-1) + 	48\alpha^4\tau^5L_2^2
				}{ m\tilde{\eta} } \sum_{p=1}^{m}\sum_{k=0}^{\tau-1} L_{3,p}^2 \|\Delta U_p^{s,k}\|_F^2 .
			\end{aligned}
		\end{equation}
		
		Adding \eqref{proof_V_descent_start},	\eqref{proof_use_lemma3}, \eqref{proof_cross_error_1},
		\eqref{proof_cross_error_2}, and \eqref{proof_Lambda_scaled}, and subtracting $F^\star$ from both sides, it gets
		\begin{equation}\label{proof_before_absorption}
			\begin{aligned}
				\Phi^{s+1}  &\leq \Phi^s -c_1 \| \widetilde{\nabla}h(\widetilde V^{s+1}) + \nabla_V f(\{U_p^{s,0}\},V^s) \|_F^2  - \frac{c_2}{m}\|\Lambda^s-\overline{\Lambda}^s\|_F^2\\
				&  -\frac{c_3}{m}  \sum_{p=1}^{m}\sum_{k=0}^{\tau-1} L_{3,p}^2 \|\Delta U_p^{s,k}\|_F^2  + c_4c_h^2-c_5 \|V^{s+1}-V^s\|_F^2,
			\end{aligned}
		\end{equation}	
		where $c_1$ and $c_4$ are defined in Theorem~\ref{Theorem}, and $c_2=\frac{1}{\tilde\eta} -6\tilde\eta L_2^2 -\frac{24\alpha^2\tau^2L_2^2}{\tilde\eta} -\frac{12\tau L_{1,a}^2}{\rho_1},$ $c_3=\frac{\rho_1}{4L_{3,a}^2}
-\tilde\eta ( \frac{\tau+1}{2}+12L_2^2\alpha^2\tau^3 ) -\frac{2\alpha^2\tau^2(\tau-1) +48\alpha^4\tau^5L_2^2}{\tilde\eta} -\frac{24\alpha^2\tau^4L_{1,a}^2}{\rho_1}$, $c_5= \frac{1-4\alpha^2\tau^2L_2^2}{2\tilde{\eta}} - \frac{L_2}{2} - \frac{2\tau L_{1,a}^2}{\rho_1}.$
The parameter conditions imply
$ \alpha\tau=\frac{\tilde\eta}{\eta_g} \leq\frac{2\tilde\eta}{3},~
L_2\tilde\eta\leq\frac{1}{10},~
\frac{2\tau\tilde\eta L_{1,a}^2}{\rho_1}<\frac{1}{36}.$
Consequently,\begin{equation*}
\begin{aligned}
\tilde\eta c_5
&=\frac12-2\alpha^2\tau^2L_2^2
-\frac{L_2\tilde\eta}{2}
-\frac{2\tau\tilde\eta L_{1,a}^2}{\rho_1}\\
&>\frac12-\frac{2}{225}-\frac{1}{20}-\frac{1}{36}
=\frac{31}{75}>0.
\end{aligned}\end{equation*}	
Thus, $c_5>0$, and omitting the nonpositive term $-c_5\|V^{s+1}-V^s\|_F^2$ yields
		\begin{equation*}
			\begin{aligned}
				\Phi^{s+1}  &\leq \Phi^s -c_1 \| \widetilde{\nabla}h(\widetilde V^{s+1}) + \nabla_V f(\{U_p^{s,0}\},V^s) \|_F^2  - \frac{c_2}{m}\|\Lambda^s-\overline{\Lambda}^s\|_F^2\\
				&  -\frac{c_3}{m}  \sum_{p=1}^{m}\sum_{k=0}^{\tau-1} L_{3,p}^2 \|\Delta U_p^{s,k}\|_F^2  + c_4c_h^2.
			\end{aligned}
		\end{equation*}	
       Following the definition of $C_s,A_s,$ and $R_V^s$ in \eqref{defCsAs} and \eqref{defRVS}, it has
        \begin{equation}\label{masterLyapunovdescent}
		\begin{aligned}
			\Phi^{s+1}-\Phi^s \leq  & -c_1R_V^s -c_2C_s -c_3 A_s +c_4 c_h^2.
		\end{aligned}
	\end{equation}

\textbf{Step 4: It bounds the average stationarity residual $\frac{1}{T}\sum_{s=1}^{T}R_V^s$.}
    
		Summing the inequality \eqref{masterLyapunovdescent} over $s=1,\cdots,T$ yields
		\begin{equation*}
			\begin{aligned}
				&\Phi^1-\Phi^{T+1} + T c_4c_h^2\geq c_3 A_s +c_1R_V^s + c_2C_s.
			\end{aligned}
		\end{equation*}
		
		Since $\Phi^{T+1}\geq0$, it has
		\begin{equation*}
			\begin{aligned}
				&\Phi^1 + T c_4c_h^2\geq \sum_{s=1}^{T}(c_1R_V^s + c_2C_s +c_3 A_s) .
			\end{aligned}
		\end{equation*}
		Since $C_s,A_s\geq0$ and $c_1,c_2,c_3>0$, it follows that
		\begin{equation*}
			\begin{aligned}
				&\frac{1}{T} \sum_{s=1}^{T} R_V^s  \leq \frac{\Phi^1}{c_1 T} + \frac{c_4}{c_1}c_h^2 .
			\end{aligned}
		\end{equation*}
		The proof is complete.
		
	\end{proof}
	
	Next, combining the preceding estimate with Lemma~\ref{LeA.4} yields a finite round bound on the stationarity residual of the local matrices.
	\begin{theorem}\label{Theorem2}
		Suppose that Assumption \ref{assumption1}-\ref{assumption3} hold and the parameter conditions of Theorem \ref{Theorem} are satisfied. Define
		\begin{equation*}
			R_U^s = \frac{1}{m}\sum_{p=1}^{m} r_p^{s,\tau}.
		\end{equation*}
		Then, 
		\begin{equation}\label{Uresidualcomplexity}
			\frac{1}{T}\sum_{s=1}^{T}R_U^s
			\leq \frac{\kappa (\Phi^1-\underline{\Phi})}{T}
			+ (\kappa c_4+b_1 )c_h^2,
		\end{equation}
		where $ \underline{\Phi}=\inf_s\Phi^s,$ $\kappa=\max\{b_2/c_1,b_3/c_2,b_4/c_3\}$, and $b_1=48\tau^3\alpha^2L_{1,a}^2$. The constants $b_2,b_3,b_4$ are defined in the proof.
	\end{theorem}
	\begin{proof}
		Let $ L_{3,i}=\min_{p }L_{3,p}>0.$ Applying Lemma~\ref{LeA.4} to $k=\tau-1$, and enlarging the resulting terms by summing over all local iterations, it obtains
		\begin{equation}\label{roundUboundstep1}
			R_U^s \leq 2L_{1,a}^2\frac{1}{m}\sum_{p=1}^{m} 	\sum_{k=0}^{\tau-1} \|V_p^{s,k}-V^s \|_F^2 + \frac{2\rho_1^2}{L_{3,i}^2}A_s.
		\end{equation}
		
		Moreover, Lemma~\ref{LeA.12} gives
		\begin{equation}\label{localcopyboundtheorem2}
			\begin{aligned}
				&\frac{1}{m}\sum_{p=1}^{m} 	\sum_{k=0}^{\tau-1} \|V_p^{s,k}-V^s \|_F^2  \leq 6\tau C_s +24\tau^3\alpha^2c_h^2
				+6\tau^3\alpha^2R_V^s
				+12\alpha^2\tau^4A_s.
			\end{aligned}
		\end{equation}
		Substituting \eqref{localcopyboundtheorem2} into
		\eqref{roundUboundstep1} yields
		\begin{equation}\label{roundwiseUboundinproof}
			R_U^s \leq b_2R_V^s +b_3C_s +b_4A_s +b_1c_h^2,
		\end{equation}
		where
		\begin{equation*}
			\begin{aligned}
				b_2 =12\tau^3\alpha^2L_{1,a}^2,~b_3=12\tau L_{1,a}^2,b_4= 24\alpha^2\tau^4L_{1,a}^2 +\frac{2\rho_1^2}{L_{3,i}^2}.
			\end{aligned}
		\end{equation*}
		
		Let
		\begin{equation*}
			 \kappa= \max \{ \frac{b_2}{c_1}, \frac{b_3}{c_2}, \frac{b_4}{c_3} \}.
		\end{equation*}
		It follows from \eqref{roundwiseUboundinproof} that
		\begin{equation*}
			R_U^s \leq \kappa ( c_1R_V^s+c_2C_s+c_3A_s ) +b_1c_h^2.
		\end{equation*}
		
		Summing \eqref{masterLyapunovdescent} from $s=1$ to $s=T$ and using $\Phi^{T+1}\geq\underline{\Phi}$ give
		\begin{equation*}
			\sum_{s=1}^{T} (
			c_1R_V^s+c_2C_s+c_3A_s )
			\leq \Phi^1-\underline{\Phi}+T c_4c_h^2,
		\end{equation*}
		where $ \underline{\Phi}=\inf_s\Phi^s.$ Therefore,
		\begin{equation*}
			\sum_{s=1}^{T}R_U^s \leq
			\kappa (\Phi^1-\underline{\Phi}) +T (\kappa c_4+b_1 )c_h^2.
		\end{equation*}
		Dividing both sides by $T$ proves \eqref{Uresidualcomplexity}.
	\end{proof}
 
	Theorems \ref{Theorem} and \ref{Theorem2} provide residual bounds for $V$ and $U_p$, respectively. Adding these bounds yields the following estimate for the combined residual.
	\begin{theorem}\label{Theorem3}
		Suppose that the conditions of Theorems~\ref{Theorem} and~\ref{Theorem2} hold. Define the overall stationarity residual as
		\begin{equation*}
			R^s = R_V^s + \frac{1}{m}\sum_{p=1}^{m} R_{U,p}^{s,\tau}.
		\end{equation*}	
		Then, 
		\begin{equation*}
			\frac{1}{T}\sum_{s=1}^{T}R^s
			\leq \frac{( \frac{1}{c_1}+\kappa )(\Phi^1-\underline{\Phi})}{T} + (\frac{c_4}{c_1}+\kappa c_4+b_1)c_h^2.
		\end{equation*}
	\end{theorem}
	\begin{proof}
		By the definition of $R^s$, it has
		\begin{equation*}
			\frac{1}{T}\sum_{s=1}^{T}R^s = \frac{1}{T}\sum_{s=1}^{T}R_V^s + \frac{1}{T}\sum_{s=1}^{T}R_U^s.
		\end{equation*}
		Combining Theorems~\ref{Theorem} and~\ref{Theorem2} gives
		\begin{equation*}
			\frac{1}{T}\sum_{s=1}^{T}R^s
			\leq ( \frac{1}{c_1}+\kappa  ) \frac{\Phi^1-\underline{\Phi}}{T}
			+ ( \frac{c_4}{c_1} +\kappa c_4+b_1 )c_h^2.
		\end{equation*}
	\end{proof}
\begin{remark}	 
Theorem \ref{Theorem3} bounds the average combined residual by a term that decays as $O(1/T)$ for fixed parameters and a remaining error term. The latter depends on the step sizes, the number of local updates, and the proximal parameter through $c_4$, $c_1$, $\kappa $, and $b_1$.
To examine the controllability of the error term, fix $\tau$, $\eta_g$, and $\rho_1$, and set $\alpha=\tilde\eta/(\eta_g\tau)$. For sufficiently small $\tilde\eta$ satisfying the parameter conditions, the coefficients satisfy $c_1=O(\tilde\eta)$, $c_4=O(\tilde\eta^2)$, $\kappa=O(1)$, and $b_1=O(\tilde\eta^2)$. Provided that the relevant Lipschitz constants and the initial auxiliary function gap remain uniformly bounded, the estimate becomes
\begin{equation*}
	\frac1T\sum_{s=1}^{T}R^s
	\leq O(\frac1{T\tilde\eta})+O(\tilde\eta)c_h^2.
\end{equation*}
Thus, decreasing the effective global step size reduces the remaining error bound, and increasing the number of rounds reduces the initial-error contribution. The average residual bound tends to zero when $\tilde\eta\to0$ and $T\tilde\eta\to\infty$.
\end{remark}
	
	\section{Numerical Experiments}\label{sec4}
	This section evaluates the proposed algorithm on the synthetic dataset and real-word multi-view datasets, focusing on optimization behavior and clustering performance.
	\subsection{Synthetic Experiments}
	\subsubsection{Synthetic Data Generation}\label{sec41}
	It generates a partitioned dataset containing $m=4$ clients and $n=1200$ samples. The data matrix of the $p$th client is constructed as $ {X}_p =  {U}_p {V} + {E}_p,$ where \begin{equation*} {X}_p\in\mathbb{R}^{d_p\times n},d_p=20,~{U}_p\in\mathbb{R}_{+}^{d_p\times r},~r=6\end{equation*} is the nonnegative orthogonal local matrix, $ {V}\in\mathbb{R}_{+}^{r\times n}$ is the shared matrix, and $ {E}_p$ denotes additive Gaussian noise. To introduce heterogeneity, it generates a common nonnegative orthogonal local matrix ${U}^{0}_p$ and a random local matrix $\widetilde{ {U}}_p$ with the same column supports. The unnormalized local trix of the $p$th client is constructed as
	\begin{equation*}\overline{ {U}}_p = (1-h) {U}^{0}_p + h\widetilde{ {U}}_p,~h\in[0,1].\end{equation*}
	Thus, $h=0$ corresponds to a homogeneous local matrix, and a larger $h$ produces stronger local matrix heterogeneity.
	\subsubsection{Evaluation Metrics}\label{sec42}
To evaluate the convergence behavior of the proposed method, it reports the normalized average stationarity residual against the number of communication rounds, measuring progress toward stationarity.

	Motivated by Theorem~\ref{Theorem3}, it defines the stationarity residual at communication round $s$ as $ R^{s}$ Its cumulative average over the first $s$ communication rounds is defined by $ \overline{R}^{s} = \frac{1}{s}\sum_{t=1}^{s}R^{t}.$
	The normalized average stationarity residual is then defined as
	\begin{equation}
		\widehat{R}^{s} = \frac{\overline{R}^{s}}
		{\max \{ \overline{R}^{1},\varepsilon \}}
		= \frac{ \frac{1}{s}\sum_{t=1}^{s}R^{t}
		}{ \max \{R^{1},\varepsilon \} },
	\end{equation}
	where $\varepsilon>0$ is a small constant introduced to avoid division by 0. A smaller value of $\widehat{R}^{s}$ indicates that the iterates are, on average, closer to satisfying the stationarity condition. The initialization at $s=0$ is excluded from the calculation.

	\subsubsection{Experimental Results}\label{sec43}
	It examines the convergence behavior of the proposed method under four heterogeneity levels, $h\in\{0,0.3,0.6,0.9\}$, with the number of local updates chosen from $\tau\in\{5,10,20,30\}$. Fig.~\ref{taures} shows the normalized average stationarity residual against the number of communication rounds. The residual decreases as communication proceeds, indicating progress toward stationarity. Similar qualitative trends are observed across the differernt heterogeneity levels. These results provide empirical evidence of stable optimization behavior in the considered settings. 
	The number of local updates $\tau$ affects the progress achieved per communication round. Larger values of $\tau$ generally yield smaller normalized average stationarity residuals at the same communication round. This trend is observed across all four heterogeneity levels, suggesting that additional local updates improve communication efficiency in terms of stationarity residual reduction.
	
	\begin{figure*}[t]
		\centering
		\subfigcapskip=-1pt
		\subfigure[$h=0$]{
			\centering
			\includegraphics[width=6.3cm]{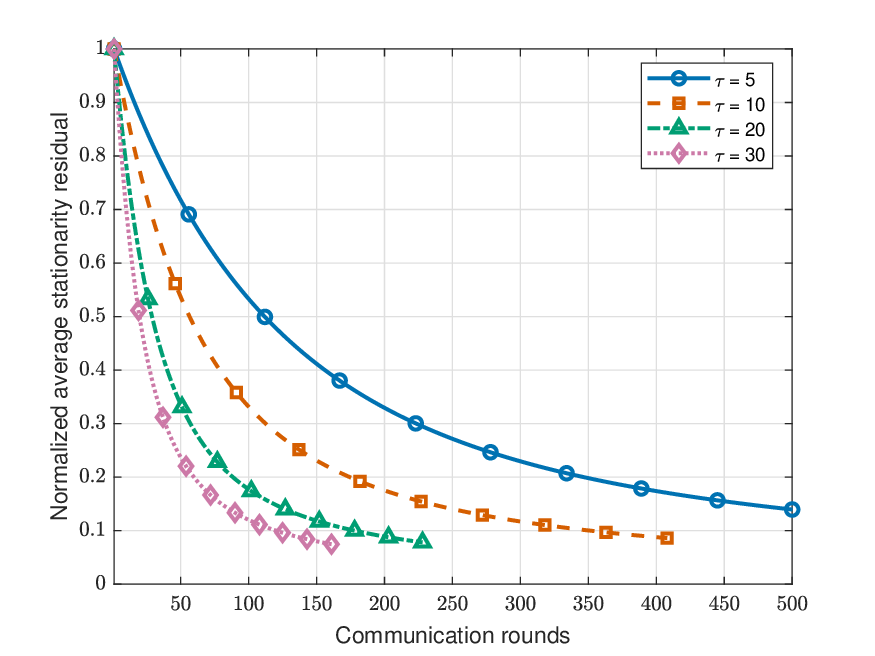}
		}\hspace{-2mm}
		\subfigcapskip=-1pt
		\subfigure[$h=0.3$]{
			\centering
			\includegraphics[width=6.3cm]{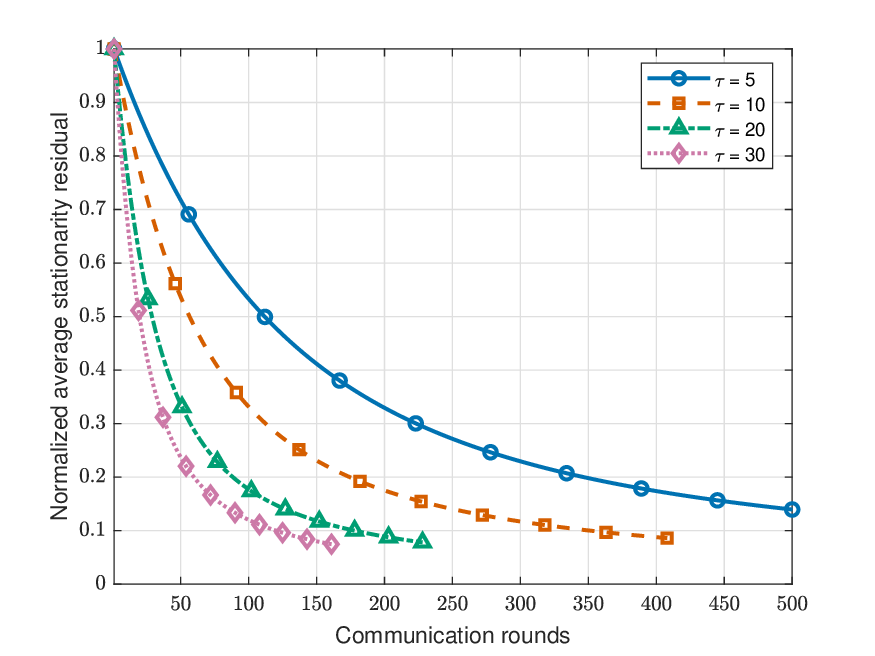}
		}\hspace{-2mm}
		\subfigcapskip=-1pt
		\subfigure[$h=0.6$]{
			\centering
			\includegraphics[width=6.3cm]{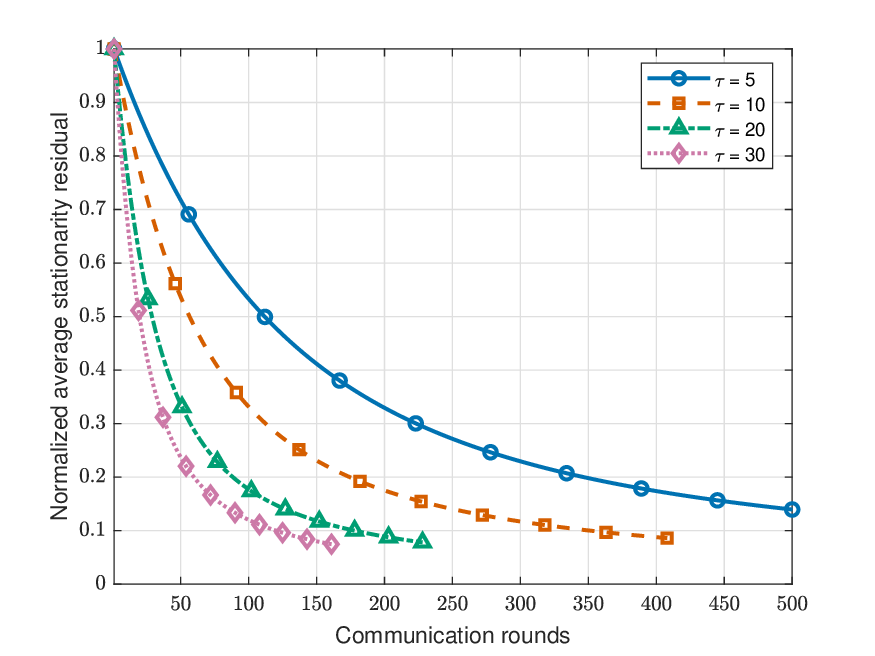}
		}\hspace{-2mm}
		\subfigcapskip=-1pt
		\subfigure[$h=0.9$]{
			\centering
			\includegraphics[width=6.3cm]{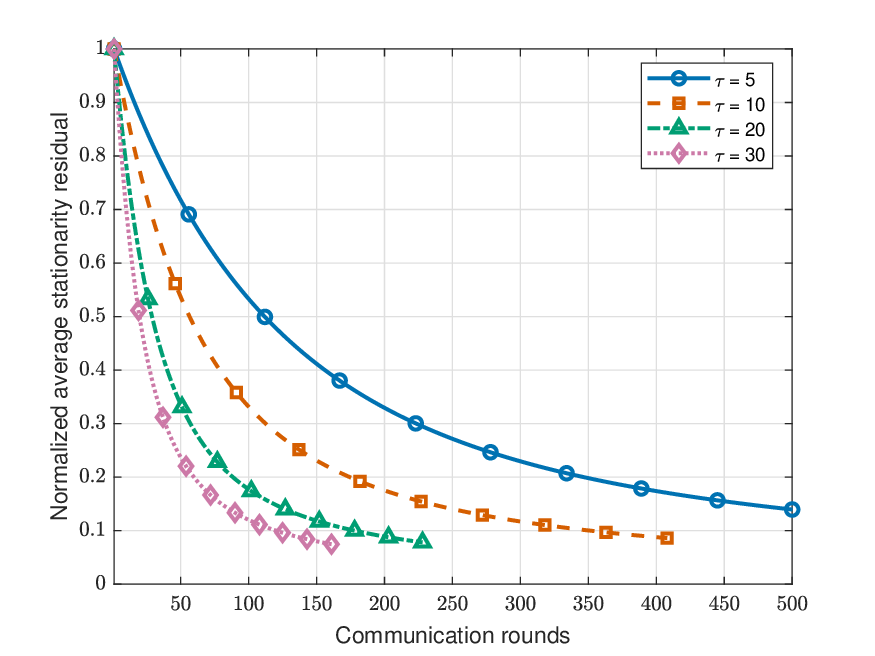}
		}\hspace{-2mm}
		\vspace{-0.1cm}
		\caption{Convergence of the normalized average stationarity residual for different numbers of local updates~$\tau$ under varying heterogeneity
			levels~$h$.}\label{taures}
	\end{figure*}
	For comparison on the same federated sparse ONMF model, it adapts three federated methods, FedMAvg \cite{wang2022federated}, FedMGS \cite{wang2022federated}, and FedVRA-U \cite{wang2023beyond}. Fig.~\ref{methres} compares the normalized average stationarity residuals of the proposed method and three adapted methods, namely Adapted FedMAvg, Adapted FedMGS, and Adapted FedVRA-U, under heterogeneity levels $h\in{0,0.3,0.6,0.9}$. Adapted FedMAvg increases initially and then remains nearly constant, while Adapted FedVRA-U increases sharply at first and then decreases gradually. The residual level at which Adapted FedMAvg stabilizes and the initial peak of Adapted FedVRA-U both increase with $h$. In contrast, the residuals of Adapted FedMGS and the proposed method decrease steadily as communication proceeds, with little change as heterogeneity increases. Both methods aggregate gradient information uploaded by clients before applying the proximal operator at the server, avoiding the need to interchange averaging and the proximal mapping for the nonsmooth regularizer. The enlarged views show that the proposed method consistently achieves lower normalized average residuals than Adapted FedMGS over the same range of communication rounds. This improvement is attributed to the gradient correction term, which effectively controls client drift.
	
	\begin{figure*}[t]
		\centering
		\subfigcapskip=-1pt
		\subfigure[$h=0$]{
			\centering
			\includegraphics[width=6.3cm]{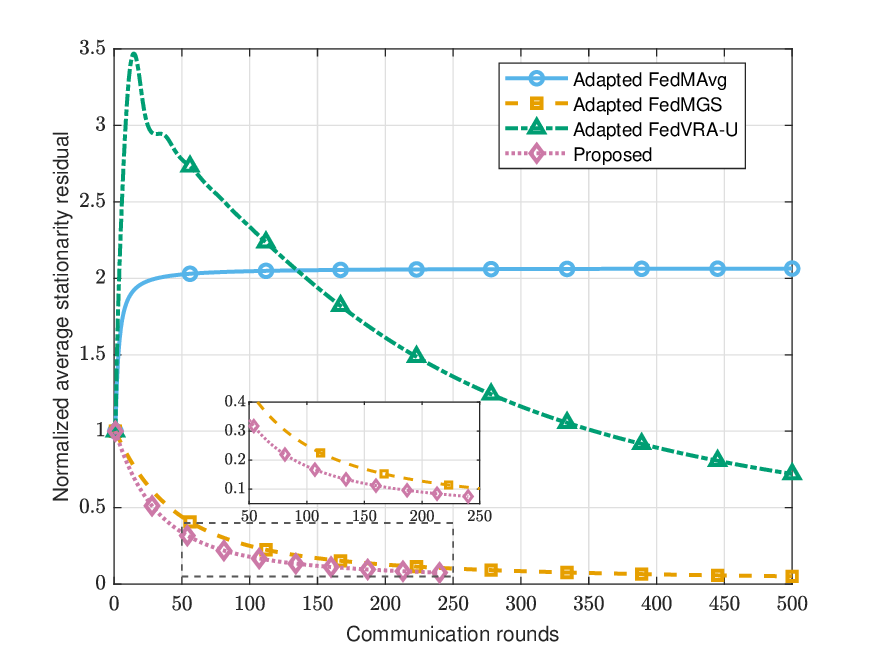}
		}\hspace{-2mm}
		\subfigcapskip=-1pt
		\subfigure[$h=0.3$]{
			\centering
			\includegraphics[width=6.3cm]{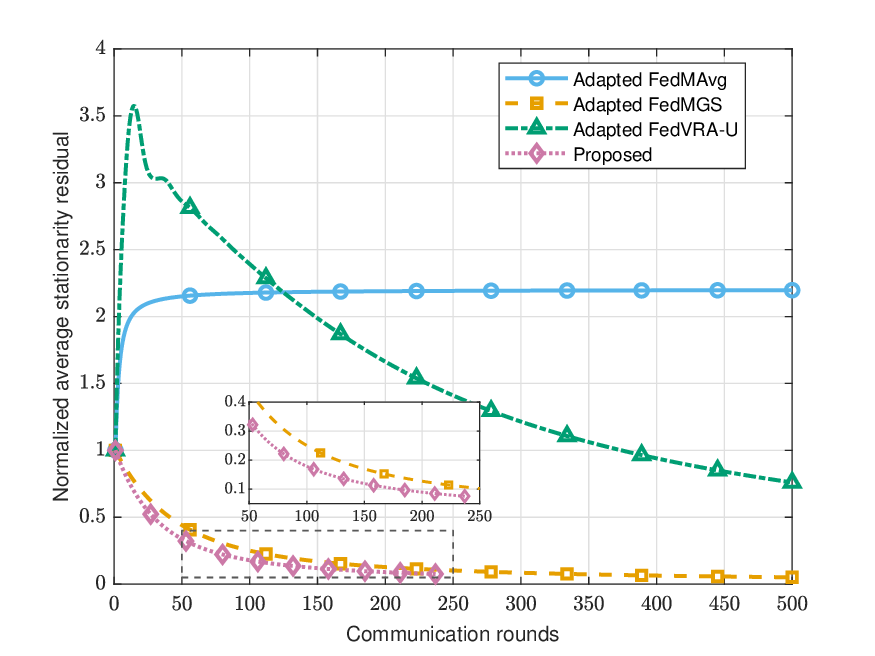}
		}\hspace{-2mm}
		\subfigcapskip=-1pt
		\subfigure[$h=0.6$]{
			\centering
			\includegraphics[width=6.3cm]{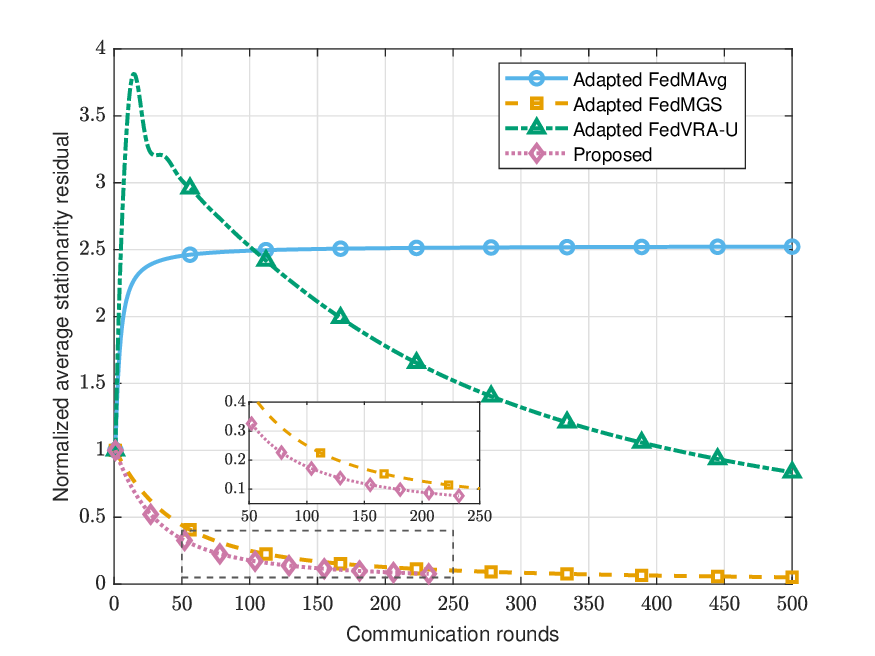}
		}\hspace{-2mm}
		\subfigcapskip=-1pt
		\subfigure[$h=0.9$]{
			\centering
			\includegraphics[width=6.3cm]{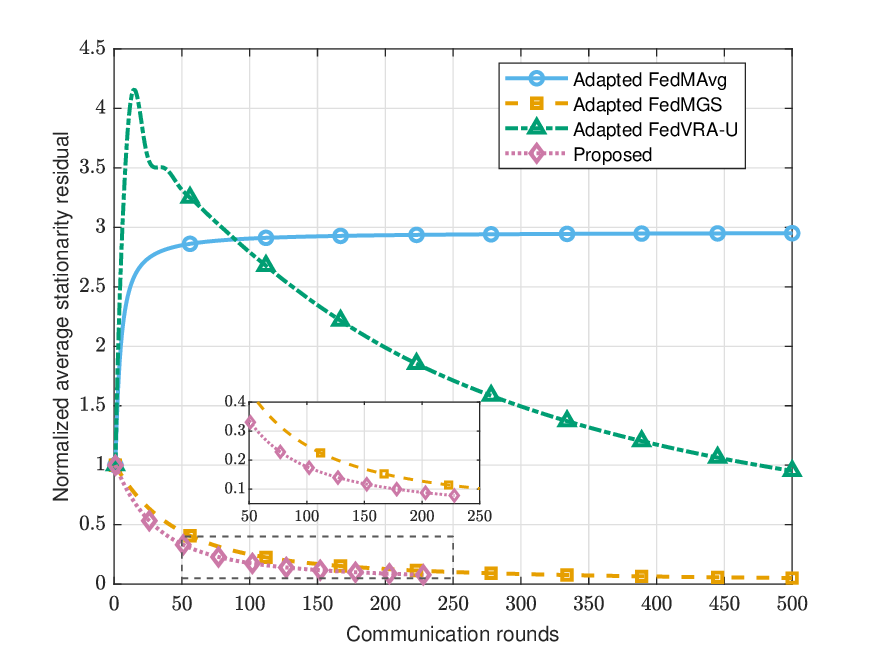}
		}\hspace{-2mm}
		\vspace{-0.1cm}
		\caption{Convergence of the normalized average stationarity residual for different federated optimization methods under varying heterogeneity levels~$h$.}\label{methres}
	\end{figure*}
	
	To evaluate the effectiveness of multiple local updates, it compares four variants, updating the shared variable $V$ at the server, updating local copies of $V$ with $\tau=1$, $\tau=10$, and the proposed method with drift correction and $\tau=10$. Fig.~\ref{abres} presents their normalized average stationarity residuals under four heterogeneity levels, $h\in\{0,0.3,0.6,0.9\} $. Fig.~\ref{abround} summarizes the communication rounds required to reach the target normalized average residual of $10^{-2}$. As shown in Fig.~\ref{abround}, the two variants with $\tau=10$ reach the target in fewer communication rounds than the server update and local copy ($\tau=1$) variants. These results demonstrate that performing multiple local updates substantially reduces the number of communication rounds needed to reach the target. Comparing the proposed method with the uncorrected local copy variant at $\tau=10$ highlights the contribution of drift correction. At $h=0$ and $h=0.3$, their residual trajectories are nearly indistinguishable. At higher heterogeneity levels, however, the proposed method achieves lower residuals over much of the displayed iteration range. At $h=0.6$ and $h=0.9$, drift correction reduces the required communication rounds from 25 to 12 and 14, corresponding to reductions of $52\%$ and $44\%$, respectively.

	\begin{figure*}[t]
		\centering
		\subfigcapskip=-1pt
		\subfigure[$h=0$]{
			\centering
			\includegraphics[width=6.3cm]{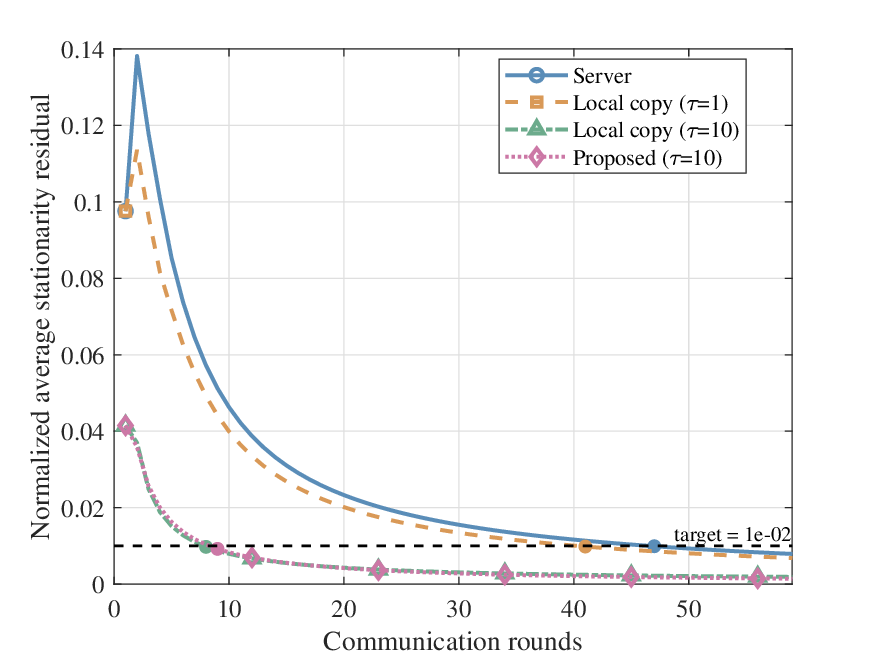}
		}\hspace{-2mm}
		\subfigcapskip=-1pt
		\subfigure[$h=0.3$]{
			\centering
			\includegraphics[width=6.3cm]{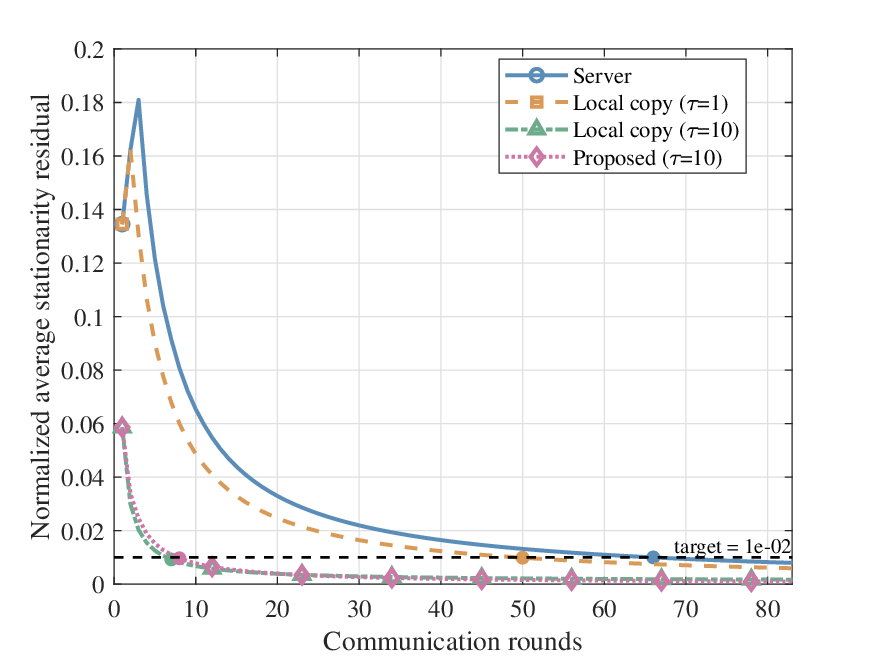}
		}\hspace{-2mm}
		\subfigcapskip=-1pt
		\subfigure[$h=0.6$]{
			\centering
			\includegraphics[width=6.3cm]{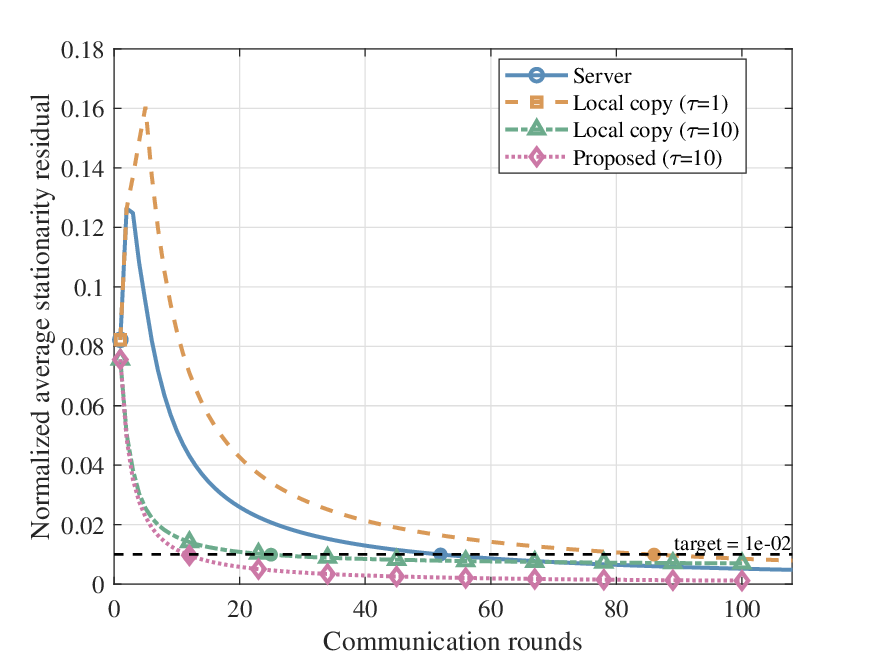}
		}\hspace{-2mm}
		\subfigcapskip=-1pt
		\subfigure[$h=0.9$]{
			\centering
			\includegraphics[width=6.3cm]{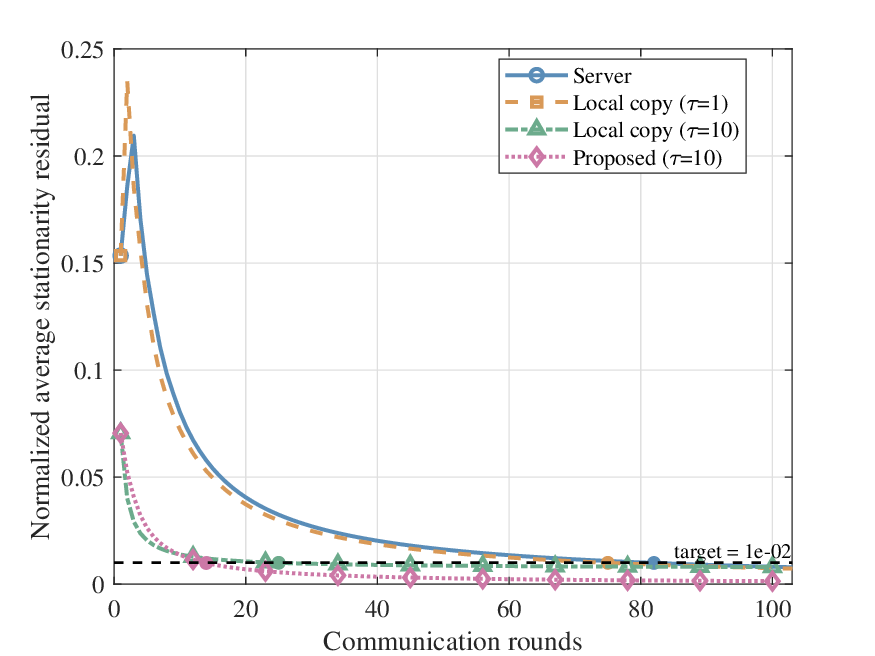}
		}\hspace{-2mm}
		\vspace{-0.1cm}
		\caption{Convergence of the normalized average stationarity residual for different
			ablation variants under varying heterogeneity levels~$h$.}\label{abres}
	\end{figure*}

	\begin{figure}
		\centering
		\includegraphics[scale=0.60]{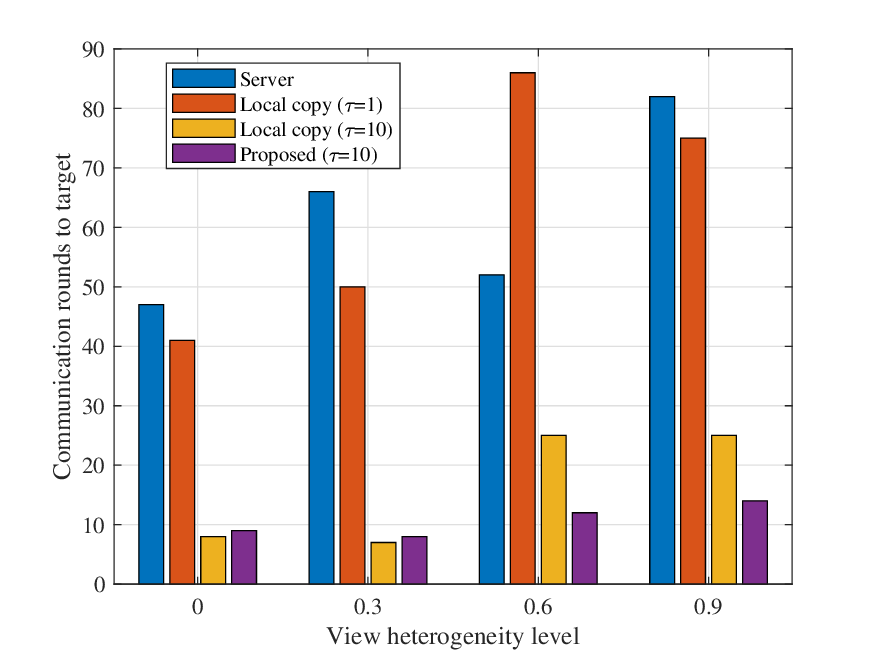} 
		\caption{Communication rounds required by different ablation variants to reach a normalized average stationarity residual of $10^{-2}$ under varying heterogeneity levels~$h$.}
		\label{abround}
	\end{figure}
    
Overall, the results on the synthetic dataset show consistent residual reduction across the diffierernt heterogeneity levels, with the proposed method attaining lower residuals than the adapted methods at comparable communication budgets. Multiple local updates reduce the communication rounds needed to reach the target residual, while drift correction provides additional gains under high heterogeneity.
    
	\subsection{Real Experiments}
	\subsubsection{Experimental Setup}
	
It evaluates the proposed method on four real-world datasets in a federated multi-view clustering setting, where each client holds one view and raw data remain local. The datasets comprise two text datasets, 3Sources\footnote{http://mlg.ucd.ie/datasets/3sources.html} and BBCSport\footnote{http://mlg.ucd.ie/datasets/bbc.html}, and two image datasets, Caltech101\footnote{https://data.caltech.edu/records/mzrjq-6wc02} and NUS-WIDE\footnote{https://github.com/ChuanbinZhang/Multi-view-datasets/blob/master/NUS-WIDE.mat}.  For these datasets, each view is normalized by its Frobenius norm. After obtaining the shared matrix $V\in\mathbb{R}^{r\times n}$, it applies $k$-means to $V^\top$. For $V$, $k$-means is performed with 10 replicates and a maximum of 300 iterations per replicate. Clustering performance is evaluated using clustering accuracy (ACC), normalized mutual information (NMI), and purity (Pur). It compares the proposed method with seven federated multi-view clustering methods, including FedMVL~\cite{huang2022efficient}, FedMAvg and FedMGS~\cite{wang2022federated}, PSFedML~\cite{li2024partially}, FNMF~\cite{dalleiger2025creating}, EFIMVC~\cite{liu2025icml}, and CeFMC~\cite{Liu2026TPAMI}. All methods are evaluates using the same input views and data preprocessing. It conductes 10 independent runs with different random seeds and report the average performance over these runs. The hyperparameters of each method are setting according to the corresponding publications.

	\subsubsection{Clustering Results}
	
	As shown in Table \ref{realres}, the proposed method achieves the highest ACC on three datasets and the highest NMI and Pur on all four datasets. The proposed method achieves the highest NMI and Pur on all four datasets and the highest ACC on three datasets, ranking second on the Caltech101. Specially, on the BBCSport, the proposed method achieves ACC, NMI, and Pur scores of $84.80\%$,$76.45\%$, and $86.00\%$, respectively. These scores exceed the strongest baseline for each metric by $3.92\%$, $10.69\%$, and $3.90\%$, respectively. These results demonstrate the competitive clustering performance of the proposed method across the tested datasets.
	
	\begin{table}[t]
		\caption{Clustering performance (\%) on real-world multi-view datasets.}
		\label{realres}
		
		\scriptsize
		\setlength{\tabcolsep}{2.8pt}
		\renewcommand{\arraystretch}{1.10}
		
		\begin{tabular}{@{}l*{12}{c}@{}}
			\toprule
			& \multicolumn{3}{c}{3Sources}
			& \multicolumn{3}{c}{BBCSport}
			& \multicolumn{3}{c}{Caltech101 }
			& \multicolumn{3}{c}{NUS-WIDE}\\
			
			\cmidrule(lr){2-4}
			\cmidrule(lr){5-7}
			\cmidrule(lr){8-10}
			\cmidrule(lr){11-13}
			
			Method
			& ACC & NMI & Pur
			& ACC & NMI & Pur
			& ACC & NMI & Pur
			& ACC & NMI & Pur\\
			\midrule
			
			FedMVL~\cite{huang2022efficient}
			& \underline{54.50} & \underline{50.32} & \underline{70.24}
			& 66.42 & 53.94 & 73.29
			& 23.28 & 9.33 & 65.07
			& 39.61 & 25.77 & 40.21\\
				
			FedMAvg~\cite{wang2022federated} 
			& 41.60 & 29.41 & 56.86
			& 70.50 & 61.17 & 76.49
			& 46.74 & 25.79 & 76.36
			& 45.79 & 34.94 & 47.80\\
			
			FedMGS~\cite{wang2022federated} 
			& 41.78 & 30.73 & 53.73
			& 73.97 & 55.45 & 74.52
			& 42.75 & 26.93 & 75.72
			& 52.28 & 39.54 & 54.30\\
			
			PSFedML~\cite{li2024partially} 
			& 39.05 & 16.18 & 40.53
			& 38.64 & 10.83 & 42.00
			& 39.94 & 21.74 & 71.24
			& 22.44 & 20.32 & 32.86\\
			
			FNMF~\cite{dalleiger2025creating} 
			& 44.38 & 33.62 & 58.97
			& 27.44 & 3.16 & 36.87
			& \textbf{51.95} & 2.55 & 54.42
			& \underline{54.98} & 41.03 & \underline{57.45}\\
			
			EFIMVC~\cite{liu2025icml} 
			& 52.90 & 47.56 & 67.87
			& \underline{80.88} & \underline{65.76} & \underline{82.10}
			& 45.55 & \underline{45.55} & 80.76
			& 54.02 & \underline{41.46} & 55.66\\
			
			CeFMC~\cite{Liu2026TPAMI} 
			& 53.11 & 44.71 & 68.34
			& 40.57 & 8.48 & 41.19
			& 47.63 & 40.43 & \underline{80.88}
			& 54.84 & 0.85 & 57.09\\
			
			Proposed 
			& \textbf{56.90} & \textbf{56.90} & \textbf{72.58}
			& \textbf{84.80} & \textbf{76.45} & \textbf{86.00}
			& \underline{49.48} & \textbf{50.40} & \textbf{87.51}
			& \textbf{56.35} & \textbf{43.34} & \textbf{57.52}\\
			
			\bottomrule
		\end{tabular}
		\footnotetext{The best and second best results are highlighted in bold and
			underlined, respectively.}
	\end{table}

	\section{Conclusion}\label{sec5}
	
	Client drift is a central challenge in federated sparse ONMF, as multiple local updates can cause client optimization trajectories to deviate from the global trajectory. In this paper, we develop a federated alternating minimization algorithm that addresses this challenge through two techniques. On the one hand, by using variable splitting, we have decoupled communication from the proximal mapping to address the noncommutativity of averaging and the proximal mapping. On the other hand, we have constructed a gradient correction term from the difference between the global and local average gradients in the previous round to control client drift during multiple local updates. Our convergence analysis has explicitly accounted for client drift and local matrix variation. Under the stated assumptions and suitable parameter conditions, we have established an upper bound on the averaged stationarity residual consisting of a decaying $O(1/T)$ term and a bounded error term. Numerical examples on synthetic and real-world datasets under different settings have demonstrated the effectiveness of the proposed method.

    There remain several important directions for future work. Firstly, incorporating stochastic gradient updates could reduce the computational cost of the proposed method. Secondly, further theoretical analysis could clarify how the maximum number of local updates affects convergence. Finally, extending the method to decentralized federated settings would broaden its applicability.
	
	\backmatter
		\bmhead{Acknowledgements}
	
This work was supported in part by the National Natural Science Foundation of China under Grants 12371306 and 12271309, in part by the Department of Science and Technology of Guangdong Province under Grant 2021CX02G450, and in part by the Postdoctoral Fellowship Program (Grade C) of China Postdoctoral Science Foundation under Grant GZC20261682.
	
		\bibliography{sn-bibliography,mybibfile}

	\begin{appendices}
		
		\section{The proof of Lemma 2}\label{secA1}
		
		
		\begin{proof}
			
			If $\tau=1$, then $V_{p}^{s,k}=V^{s}$.
            
			If $\tau\ge 2$,\begin{equation}\label{preVac}
				\begin{aligned}
					\widehat{V}_p^{s,k+1}	&= \widehat{V}_p^{s,k} - \alpha ( \nabla_{V_p} f_p(U_p^{s,k+1},V_p^{s,k}) + C_p^{s} )\\
					&=\widehat{V}_p^{s,0}- \alpha \sum_{l=0}^{k} (\nabla_{V_p} f_p(U_p^{s,l+1},V_p^{s,l}) + C_p^{s}).
				\end{aligned}
			\end{equation}
			Since $\widehat{V}_p^{s,0}=\mathcal{P}_{\tilde{\eta}\mu} (\overline {V}^{s})=V^{s} $ and ${V}_p^{s,k+1}=\mathcal{P}_{(k+1)\alpha\mu}(\widehat{V}_p^{s,k+1})$, applying \eqref{preVac} it has
			\begin{equation}
				\begin{aligned}\label{A9}
					\|V_{p}^{s,k+1}-V^{s}\|_{F}^{2} &=\|\mathcal{P}_{(k+1)\alpha\mu}(\widehat{V}_p^{s,k+1})-\mathcal{P}_{\tilde{\eta}\mu}(\overline {V}^{s})\|_{F}^{2}\\
					&=\|\mathcal{P}_{(k+1)\alpha\mu}(V^s- \alpha \sum_{l=0}^{k} (\nabla_V f_p(U_p^{s,l+1},V_p^{s,l})+ C_p^{s}))-V^s\|_{F}^{2}.
				\end{aligned}
			\end{equation}
			Define $H_p^s=\nabla_V f_p(U_p^{s,0},V^s),H^s=\frac{1}{m}\sum_{p=1}^{m}H_p^s.$
			The right side of \eqref{A9} can be interpreted as an inexact projected gradient descent, executed from the starting point $V^{s}$ using the inexact gradient $- \alpha \sum_{l=0}^{k} (\nabla_V f_p(U_p^{s,l+1},V_p^{s,l})+ C_p^{s})$. This inspires us to bound \eqref{A9} using the deviation from the exact projected gradient descent step with the same initial point $\widetilde{V}_{pgd}=\mathcal{P}_{(k+1){\eta}}(V^{s}-(k+1)\alpha H^s)$.
			In this way, 
			\begin{equation}\label{A.10}
				\begin{aligned}
					&\quad\|V_{p}^{s,k+1}-V^s\|_{F}^{2}\\
					&=\|\mathcal{P}_{(k+1)\alpha\mu}(V^s- \alpha \sum_{l=0}^{k} (\nabla_{V_p} f_p(U_p^{s,l+1},V_p^{s,l}) +C_p^{s}))-\widetilde{V}_{pgd}+\widetilde{V}_{pgd}-{V}^{s}\|_{F}^{2}\\
					&\le \underbrace{2\|\mathcal{P}_{(k+1)\alpha\mu}(V^{s}- \alpha \sum_{l=0}^{k} (\nabla_{V_p} f_p(U_p^{s,l+1},V_p^{s,l})+ C_p^{s}))-\widetilde{V}_{pgd}\|_{F}}_{(\text{III})} +\underbrace{2\|\widetilde{V}_{pgd}-{V}^{s}\|_{F}^{2}}_{(\text{IV})}.
				\end{aligned}
			\end{equation}
			Since 
			\begin{equation*}
				\begin{aligned}\widetilde{V}_{pgd}= \arg\min_{Z}\frac{1}{2}\|V^s-(k+1)\alpha H^s-Z\|_{F}^{2}+(k+1)\alpha h(Z),\end{aligned}
			\end{equation*} so $\widetilde{V}_{pgd}$ satisfies the optimality condition\begin{equation*}
				\begin{aligned}(k+1)\alpha(\tilde{\nabla}h(\widetilde{V}_{pgd})+H^s)+\widetilde{V}_{pgd}-\mathcal{P}_{\tilde{\eta}\mu} (\overline {V}^{s})=0,\end{aligned}
			\end{equation*} for some $\tilde{\nabla}h(\widetilde{V}_{pgd})\in\partial h(\widetilde{V}_{pgd})$. This implies that
			\begin{equation*}
				\begin{aligned}
					\widetilde{V}_{pgd}-V^s= -(k+1)\alpha(\tilde{\nabla}h(\widetilde{V}_{pgd})+H^s).
				\end{aligned}
			\end{equation*}
			Note that the server step size is $\tilde{\eta}$ and not $(k+1)\alpha$. To correct for this, it uses the iterate $\widetilde{V}^{s+1}$ defined in \eqref{virtual}, which is the same as $\widetilde{V}_{pgd}$ except that the step size $(k+1)\alpha$ is replaced by $\tilde{\eta}$. Following similar steps as above, the optimality condition for $\widetilde{V}^{s+1}$ implies that
			\begin{equation*}
				\begin{aligned}
					\widetilde{V}^{s+1}-V^s= -\tilde{\eta}(\tilde{\nabla}h(\widetilde{V}^{s+1})+H^s),
				\end{aligned}
			\end{equation*}
			for some $\tilde{\nabla}h(\widetilde{V}^{s+1})\in\partial h(\widetilde{V}^{s+1})$. For (\text{IV}), according to Assumption \ref{assumption3}, it has
			\begin{equation*}
				\begin{aligned}
					(\text{IV})&=2\|(k+1)\alpha(\tilde{\nabla}h(\widetilde{V}_{pgd})+H^s)\|_{F}^{2}\\
					&=2\|(k+1)\alpha(\tilde{\nabla}h(\widetilde{V}_{pgd})+H^s-\tilde{\nabla}h(\widetilde{V}^{s+1})+H^s) +(k+1)\alpha(\tilde{\nabla}h(\widetilde{V}^{s+1})+H^s)\|_{F}^{2}\\
					&=2\|(k+1)\alpha(\tilde{\nabla}h(\widetilde{V}_{pgd})-\tilde{\nabla}h(\widetilde{V}^{s+1}))+(k+1)\alpha(\tilde{\nabla}h(\widetilde{V}^{s+1})+H^s)\|_{F}^{2}\\
					&\le 4\tau^{2}\alpha^{2} 4B_{g}^{2}+4\tau^{2}\alpha^{2}\|\tilde{\nabla}h(\widetilde{V}^{s+1})+H^s\|_{F}^{2}.
				\end{aligned}
			\end{equation*}
			By the definition of $\widetilde{V}_{pgd}$ and the nonexpansiveness of the proximal operator, it has
			\begin{equation*}
				\begin{aligned}
					(\text{III})&=2\|\mathcal{P}_{(k+1)\alpha\mu}(V^s- \alpha \sum_{l=0}^{k} (\nabla_{V_p} f_p(U_p^{s,l+1},V_p^{s,l})+ C_p^{s}))-\mathcal{P}_{(k+1){\alpha}}(V^s-(k+1)\alpha H^s)\|_{F}^{2}\\
					&\le 2\|\alpha \sum_{l=0}^{k} (\nabla_{V_p} f_p(U_p^{s,l+1},V_p^{s,l})+ C_p^{s})-(k+1)\alpha H^s\|_{F}^{2}\\
					&\le 2\alpha^2\| \sum_{l=0}^{k} (\nabla_{V_p} f_p(U_p^{s,l+1},V_p^{s,l})+ C_p^{s}-H^s)\|_{F}^{2}.\\
				\end{aligned}
			\end{equation*}
			According to the definition of $\Lambda_p^s$ and $\Lambda_{a}^s$, the correction term satisfies
			\begin{equation*}\label{correction_lambda_relation}
				H_p^s+C_p^s-H^s=	\frac{1}{\alpha\tau}
				(\Lambda_p^s-\Lambda_{a}^s).
			\end{equation*}
			It has
			\begin{equation*}
				\begin{aligned}
					(\text{III})&\le  2\alpha^2\| \sum_{l=0}^{k} (\nabla_{V_p} f_p(U_p^{s,l+1},V_p^{s,l})-h_{p}^s+h_{p}^s+ C_p^{s}-H^s)\|_{F}^{2}\\
					&=2\alpha^2\| \sum_{l=0}^{k} (\nabla_{V_p} f_p(U_p^{s,l+1},V_p^{s,l})-h_{p}^s)+(k+1)(h_{p}^s+ C_p^{s}-H^s)\|_{F}^{2}\\
					&\le 4\alpha^2\| \sum_{l=0}^{k} (\nabla_{V_p} f_p(U_p^{s,l+1},V_p^{s,l})-h_{p}^s)\|_{F}^{2}+4	\alpha^2\frac{(k+1)^{2}}{\alpha^2\tau^2}
					\|\Lambda_p^s-\Lambda_{a}^s\|_{F}^2.
				\end{aligned}
			\end{equation*}
			The first term contains both the $V$ drift and the $U$ drift. Specifically,
			\begin{equation*}\label{gradient_drift_decomposition}
				\begin{aligned}
					&\quad\nabla_{V_p} f_p(U_p^{s,l+1},V_p^{s,l})-H_p^s\\
					&=\underbrace{\nabla_{V_p} f_p(U_p^{s,l+1},V_p^{s,l})-\nabla_V f_p(U_p^{s,l+1},V^s)}_{\text{$V$ drift}}\\
                    &+\underbrace{\nabla_V f_p(U_p^{s,l+1},V^s)-\nabla_V f_p(U_p^{s,0},V^s)}_{\text{$U$ drift}}.
				\end{aligned}
			\end{equation*}
			Using the Lipschitz continuity gives
			\begin{equation*}\label{gradient_drift_bound}
				\begin{aligned}
					\|\nabla_{V_p} f_p(U_p^{s,l+1},V_p^{s,l})-H_p^s\|_F^2\leq 2L_2^2\|V_p^{s,l}-V^s\|_F^2+2L_{3,p}^2\|U_p^{s,l+1}-U_p^{s,0}\|_F^2.
				\end{aligned}
			\end{equation*}
			
			Using $\|\sum_{l=0}^{k}A_l\|_F^2\leq(k+1)\sum_{l=0}^{k}\|A_l\|_F^2,$ it obtains
			\begin{equation*}\label{modified_term_I_bound}
				\begin{aligned}
					(\text{III})\leq &8(k+1)\alpha^2L_2^2\sum_{l=0}^{k}\|V_p^{s,l}-V^s\|_F^2\\
					&+8(k+1)\alpha^2L_{3,p}^2\sum_{l=0}^{k}\|U_p^{s,l+1}-U_p^{s,0}\|_F^2 +4(\frac{k+1}{\tau})^2\|\Lambda_p^s-\Lambda_{a}^s\|_F^2.
				\end{aligned}
			\end{equation*}
			Now, plugging (\text{III}) and (\text{IV}) into \eqref{A.10} yields
			
			\begin{equation}\label{A.19}
				\begin{aligned}
					\|V_{p}^{s,k+1}-V^s\|_{F}^{2}
					&\le 8(k+1)\alpha^2L_2^2\sum_{l=0}^{k}\|V_p^{s,l}-V^s\|_F^2+	4\|\Lambda_p^s-\Lambda_{a}^s\|_F^2+16\tau^2\alpha^2c_h^2	\\
					&+ 4\tau^2\alpha^2\|\widetilde\nabla h(\widetilde V^{s+1})+H^s\|_F^2 +8(k+1)\alpha^2L_{3,p}^2\sum_{l=0}^{k}\|U_p^{s,l+1}-U_p^{s,0}\|_F^2.
				\end{aligned}
			\end{equation}
			To get a recursion for the drift error from \eqref{A.19}, let \begin{equation*}
				\begin{aligned}B^{s}&=4\|\Lambda_p^s-\Lambda_{a}^s\|_F^2+16\tau^2\alpha^2c_h^2+
			4\tau^2\alpha^2\|\widetilde\nabla g(\widetilde V^{s+1})+H^s\|_F^2,\\
            ~E_{U}^{k}&=8(k+1)\alpha^2L_{3,p}^2
			\sum_{l=0}^{k}\|U_p^{s,l+1}-U_p^{s,0}\|_F^2,\end{aligned}
			\end{equation*} and define $\mathcal{S}_{p,k}^{s}=\sum_{l=0}^{k}\|V_p^{s,l}-V^s\|_{F}^{2}$.
			\begin{equation*}
				\begin{aligned}
					\mathcal{S}_{p,k+1}^{s}-\mathcal{S}_{p,k}^{s}\le  8\tau\alpha^{2}L_{V}^{2}\mathcal{S}_{p,k}^{s}+B^{s}+E_{U}^{k}.
				\end{aligned}
			\end{equation*}
			Under the step-size condition $\alpha\tau L_2\leq\frac{1}{\sqrt{20}}$, it has $8\tau\alpha^2L_2^2\leq\frac{2}{5\tau}.$ Consequently, $(1+8\tau\alpha^2L_2^2)^\tau\le\exp (\frac{2}{5} )\approx 1.492$,
			it concludes that
			\begin{equation}\label{A.21}
				\begin{aligned}
					\mathcal S_{p,\tau-1}^s\leq1.492\tau A_p^s+1.492\sum_{k=0}^{\tau-2}E_{U}^{k}.
				\end{aligned}
			\end{equation}
			Summing \eqref{A.21} over all the client, it has	
			\begin{equation*}\label{A.211}
				\begin{aligned}
					\sum_{k=0}^{\tau-1}\sum_{p=1}^{m}\|V_p^{s,k}-V^s\|_F^2 
					&\leq 6\tau \|\Lambda^s-\overline{\Lambda}^s\|_F^2 + 24m\tau^3\alpha^2c_h^2 \\
					&+ 6m\tau^3\alpha^2 \| \widetilde{\nabla}h(\widetilde V^{s+1}) + \nabla_Vf(\{U_p^{s,0}\},V^s) \|_F^2 \\
					&+ 12\alpha^2\tau^4 \sum_{p=1}^{m}\sum_{k=0}^{\tau-1} L_{3,p}^2 \|\Delta U_p^{s,k}\|_F^2.
				\end{aligned}
			\end{equation*}
		\end{proof}
	
		\section{The proof of Lemma 3}\label{secA2}
	
	\begin{proof}
		
		By the definition of $\Lambda_{p}^{s+1}$ and $\Lambda_{a}^{s+1}$, it has
		\begin{equation}\label{Lambda_direct_variance_bound}
			\begin{aligned}
				\frac{1}{m}\|\Lambda^{s+1}-\overline{\Lambda}^{s+1}\|_{F}^{2}\le\frac{\alpha^2\tau}{m}\sum_{p=1}^{m}\sum_{k=0}^{\tau-1}\|\nabla_V f_p(U_p^{s+1,0},V^{s+1})-\nabla_{V_p} f_p(U_p^{s,k+1},V_p^{s,k})\|_F^2,
			\end{aligned}
		\end{equation}
		where the inequality follows from the fact that $\sum_{p=1}^{m}\|Y_p-\frac{1}{m}\sum_{q=1}^{m}Y_q\|_F^2\leq\sum_{p=1}^{m}\|Y_p\|_F^2$ and $\|\sum_{k=0}^{\tau-1}Y_k\|_F^2\leq\tau\sum_{k=0}^{\tau-1}\|Y_k\|_F^2$ and
		\begin{equation*}
			\begin{aligned}
				\Lambda_p^{s+1}-\Lambda_{a}^{s+1} 
				&=\alpha\sum_{k=0}^{\tau-1}(\nabla_V f_p(U_p^{s+1,0},V^{s+1})-
				\nabla_{V_p} f_p(U_p^{s,k+1},V_p^{s,k})\\
				&-\frac{1}{m}\sum_{p=1}^{m}(\nabla_V f_p(U_p^{s+1,0},V^{s+1})-\nabla_{V_p} f_p(U_p^{s,k+1},V_p^{s,k}))).
			\end{aligned}
		\end{equation*}
		
		For every client $p$ and local iteration $k$, adding and subtracting $\nabla_V f_p(U_p^{s+1,0},V^s)$ and $\nabla_V f_p(U_p^{s,k+1},V^s)$ gives
		\begin{equation*}
			\begin{aligned}
				&\quad \nabla_V f_p(U_p^{s+1,0},V^{s+1})-\nabla_{V_p} f_p(U_p^{s,k+1},V_p^{s,k})\\
				&=\nabla_V f_p(U_p^{s+1,0},V^{s+1})-\nabla_V f_p(U_p^{s+1,0},V^s)+\nabla_V f_p(U_p^{s+1,0},V^s)\\
				&-\nabla_V f_p(U_p^{s,k+1},V^s)+\nabla_V f_p(U_p^{s,k+1},V^s)-\nabla_{V_p} f_p(U_p^{s,k+1},V_p^{s,k}). 
			\end{aligned}
		\end{equation*}
		
		Using the Lipschitz continuity and the inequality $\|x+y+z\|_F^2\leq2\|x\|_F^2+4\|y\|_F^2+4\|z\|_F^2,$ it obtains
		\begin{equation}\label{direct_gradient_difference_bound}
			\begin{aligned}
				&\quad\|\nabla_V f_p(U_p^{s+1,0},V^{s+1})-\nabla_{V_p} f_p(U_p^{s,k+1},V_p^{s,k})\|_F^2\\
				&\leq2L_2^2\|V^{s+1}-V^s\|_F^2+4L_2^2\|V_p^{s,k}-V^s\|_F^2 +4L_{3,p}^2\|U_p^{s+1,0}-U_p^{s,k+1}\|_F^2.
			\end{aligned}
		\end{equation}
		
		Under the condition $U_p^{s+1,0}=U_p^{s,\tau}$, it follows that
		\begin{equation*}
			\begin{aligned}
				U_p^{s+1,0}-U_p^{s,k+1}=\sum_{l=k+1}^{\tau-1}(U_p^{s,l+1}-U_p^{s,l}).
			\end{aligned}
		\end{equation*}
		
		By the Cauchy-Schwarz inequality,
		\begin{equation}\label{direct_terminal_U_bound}
			\begin{aligned}
				\sum_{k=0}^{\tau-1}\|U_p^{s+1,0}-U_p^{s,k+1}\|_F^2 	&\leq\sum_{k=0}^{\tau-1}(\tau-k-1)\sum_{l=k+1}^{\tau-1}\|U_p^{s,l+1}-U_p^{s,l}\|_F^2\\
				&\leq\frac{\tau(\tau-1)}{2}\sum_{l=0}^{\tau-1}\|U_p^{s,l+1}-U_p^{s,l}\|_F^2.
			\end{aligned}
		\end{equation}
		
		Substituting \eqref{direct_gradient_difference_bound} and \eqref{direct_terminal_U_bound} into \eqref{Lambda_direct_variance_bound} yields
		\begin{equation*}\label{Lambda_recursion_compact_direct}
			\begin{aligned}
				&\quad\frac{1}{m}\|\Lambda^{s+1}-\overline{\Lambda}^{s+1}\|_F^2-2\alpha^2\tau^2L_2^2\|V^{s+1}-V^s\|_F^2\\
				&\leq\frac{4\alpha^2\tau L_2^2}{m}\sum_{p=1}^{m}\sum_{k=0}^{\tau-1}\|V_p^{s,k}-V^s\|_F^2 +\frac{2\alpha^2\tau^2(\tau-1)}{m}\sum_{p=1}^{m}\sum_{k=0}^{\tau-1}L_{3,p}^2\|\Delta U_p^{s,k}\|_F^2.
			\end{aligned}
		\end{equation*}
		
		Applying Lemma~\ref{LeA.12} to the local $V$ drift term gives
		\begin{equation*}\label{Lambda_recursion_expanded_direct}
			\begin{aligned}
				&\quad\frac{1}{m}\|\Lambda^{s+1}-\overline{\Lambda}^{s+1}\|_F^2-2\alpha^2\tau^2L_2^2\|V^{s+1}-V^s\|_F^2\\
				&\leq\frac{24\alpha^2\tau^2L_2^2}{m}\|\Lambda^s-\overline{\Lambda}^s\|_F^2 +24\alpha^4\tau^4L_2^2\|\widetilde{\nabla}h(\widetilde V^{s+1})+\nabla_V f(\{U_p^{s,0}\},V^s)\|_F^2\\
				&+96\alpha^4\tau^4L_2^2c_h^2+\frac{2\alpha^2\tau^2(\tau-1)+48\alpha^4\tau^5L_2^2}{m}\sum_{p=1}^{m}\sum_{k=0}^{\tau-1}L_{3,p}^2\|\Delta U_p^{s,k}\|_F^2.
			\end{aligned}
		\end{equation*} The proof is complete.
	\end{proof}
	
		\section{The proof of Lemma 4}\label{secA3}
	
	\begin{proof}
		Under Assumption \ref{assumption2}, it ensure that $U_p^{s,k+1}$ is a global minimizer of subproblem~\eqref{Um1}. Since $U_p^{s,k}\in\mathcal O_p$ is feasible for this subproblem and the proximal term vanishes at $U_p^{s,k}$, it has
		\begin{equation}
			\begin{aligned}
				&f_p(U_p^{s,k+1},V_p^{s,k}) +\frac{\rho_1}{2} \|\Delta U_p^{s,k}\|_F^2\leq f_p(U_p^{s,k},V_p^{s,k}).
			\end{aligned}
		\end{equation}
		Therefore,
		\begin{equation}\label{U_step_descent_difference}
			\begin{aligned}
				&f_p(U_p^{s,k+1},V_p^{s,k}) - f_p(U_p^{s,k},V_p^{s,k})\leq -\frac{\rho_1}{2} \|\Delta U_p^{s,k}\|_F^2.
			\end{aligned}
		\end{equation}
		Next, 
		\begin{equation}\label{U_cross_eval_decomposition}
			\begin{aligned}
				&\quad f_p(U_p^{s,k+1},V^{s+1}) - f_p(U_p^{s,k},V^{s+1}) \\
				&= f_p(U_p^{s,k+1},V_p^{s,k}) - f_p(U_p^{s,k},V_p^{s,k})+( f_p(U_p^{s,k+1},V^{s+1}) \\
				& - f_p(U_p^{s,k+1},V_p^{s,k}))-( f_p(U_p^{s,k},V^{s+1}) - f_p(U_p^{s,k},V_p^{s,k})).
			\end{aligned}
		\end{equation}
		Let $J_p(U) = f_p(U,V^{s+1})-f_p(U,V_p^{s,k}).$ Then the last two bracketed terms of \eqref{U_cross_eval_decomposition} are $J_p(U_p^{s,k+1})-J_p(U_p^{s,k})$. Since the feasible set induced by $U_p\geq 0$ and $U_p^\top U_p=I$ is not convex, the Euclidean line segment between
		$U_p^{s,k}$ and $U_p^{s,k+1}$ may not remain feasible. Hence, it uses the integral form along the Euclidean segment. Since $f_p$ is continuously differentiable in an open neighborhood containing this segment, it has
		\begin{equation*}
			\begin{aligned}
				&J_p(U_p^{s,k+1})-J_p(U_p^{s,k}) = \int_0^1 \langle \nabla_U J_p ( U_p^{s,k} +t(U_p^{s,k+1}-U_p^{s,k}) ), U_p^{s,k+1}-U_p^{s,k} \rangle dt .
			\end{aligned}
		\end{equation*}
		By the Cauchy--Schwarz inequality, it follows that
		\begin{equation*}
			\begin{aligned}
				& |J_p(U_p^{s,k+1})-J_p(U_p^{s,k}) | \leq \sup_{t\in[0,1]} \| \nabla_U J_p ( U_p^{s,k} +t(U_p^{s,k+1}-U_p^{s,k}) ) \|_F  \|U_p^{s,k+1}-U_p^{s,k} \|_F .
			\end{aligned}
		\end{equation*}
		Moreover, since $ \nabla_U J_p(U) = \nabla_U f_p(U,V^{s+1}) - \nabla_U f_p(U,V_p^{s,k}),$ the cross-Lipschitz continuity of $\nabla_U f_p$ with respect to $V$ gives
		\begin{equation*}
			\begin{aligned}
				\|\nabla_U J_p(U) \|_F
				&= \| \nabla_U f_p(U,V^{s+1}) - \nabla_U f_p(U,V_p^{s,k}) \|_F \leq L_{1,p} \|V^{s+1}-V_p^{s,k} \|_F .
			\end{aligned}
		\end{equation*}
		Therefore,
		\begin{equation}\label{U_cross_eval_error}
			\begin{aligned}
				&( f_p(U_p^{s,k+1},V^{s+1}) - f_p(U_p^{s,k+1},V_p^{s,k})) -(  f_p(U_p^{s,k},V^{s+1}) - f_p(U_p^{s,k},V_p^{s,k}))\\
				&\leq L_{1,p} \|V^{s+1}-V_p^{s,k}\|_F \|U_p^{s,k+1}-U_p^{s,k}\|_F.
			\end{aligned}
		\end{equation}
		Applying Young's inequality gives
		\begin{equation}\label{U_cross_eval_young}
			\begin{aligned}
				& L_{1,p} \|V^{s+1}-V_p^{s,k}\|_F \|U_p^{s,k+1}-U_p^{s,k}\|_F \\ &\leq \frac{\rho_1}{4} \|\Delta U_p^{s,k}\|_F^2 + \frac{L_{1,p}^2}{\rho_1} \|V^{s+1}-V_p^{s,k}\|_F^2.
			\end{aligned}
		\end{equation}
		Combining \eqref{U_step_descent_difference}, \eqref{U_cross_eval_decomposition}, \eqref{U_cross_eval_error}, and \eqref{U_cross_eval_young}, we obtain
		\begin{equation}\label{U_single_cross_descent}
			\begin{aligned}
				&f_p(U_p^{s,k+1},V^{s+1}) - f_p(U_p^{s,k},V^{s+1})  \\
                &\leq -\frac{\rho_1}{4} \|\Delta U_p^{s,k}\|_F^2 + \frac{L_{1,p}^2}{\rho_1} \|V^{s+1}-V_p^{s,k}\|_F^2 .
			\end{aligned}
		\end{equation}
		
		Summing \eqref{U_single_cross_descent} over $k=0,\cdots,\tau-1$ gives
		\begin{equation*}
			\begin{aligned}
				&f_p(U_p^{s,\tau},V^{s+1}) - f_p(U_p^{s,0},V^{s+1})  \\
                &\leq -\frac{\rho_1}{4} \sum_{k=0}^{\tau-1} \|\Delta U_p^{s,k}\|_F^2 + \frac{1}{\rho_1} \sum_{k=0}^{\tau-1} L_{1,p}^2 \|V^{s+1}-V_p^{s,k}\|_F^2 .
			\end{aligned}
		\end{equation*}
		Using $U_p^{s+1,0}=U_p^{s,\tau}$, summing over all clients,
		and multiplying by $1/m$, it has
		\begin{equation}\label{57}
			\begin{aligned}
				&\frac1m\sum_{p=1}^{m}f_p(U_p^{s+1,0},V^{s+1}) - \frac1m\sum_{p=1}^{m}
				f_p(U_p^{s,0},V^{s+1}) \\
				&\leq -\frac{\rho_1}{4m} \sum_{p=1}^{m}\sum_{k=0}^{\tau-1} \|\Delta U_p^{s,k}\|_F^2  + \frac{1}{m\rho_1} \sum_{p=1}^{m}\sum_{k=0}^{\tau-1} L_{1,p}^2 \|V^{s+1}-V_p^{s,k}\|_F^2 .
			\end{aligned}
		\end{equation}
		Since $h(V^{s+1})$ is independent of $\{U_p\}$, the above inequality proves \eqref{roundwise_U_descent_extended}.
		
		Finally, because $L_{3,p}\leq L_{3,a}$, it has $\|\Delta U_p^{s,k}\|_F^2 \geq \frac{L_{3,p}^2}{L_{3,a}^2} \|\Delta U_p^{s,k}\|_F^2.$
		Substituting this inequality into \eqref{57} yields \eqref{roundwise_U_descent_extended}. This completes the proof.
	\end{proof}	
	\section{The proof of Lemma 5}\label{secA4}
	\begin{proof}
		Under Assumption \ref{assumption2}, it ensures that $U_p^{s,k+1}$ is a global minimizer, and hence a local minimizer, of subproblem~\eqref{Um1}. Since the objective is continuously differentiable and $\mathcal O_p$ is closed, the generalized Fermat rule gives
		\begin{equation*}\label{eq:exact-U-optimality}
			0\in\nabla_{U_p}f_p(U_p^{s,k+1},V_p^{s,k})+\rho_1(U_p^{s,k+1}-U_p^{s,k})+N_{\mathcal O_p}(U_p^{s,k+1}).
		\end{equation*}
		Therefore, there exists $Z_p^{s,k+1}\in N_{\mathcal O_p}(U_p^{s,k+1})$ such that
		\begin{equation*}\label{eq:exact-U-normal}
			\nabla_{U_p}f_p(U_p^{s,k+1},V_p^{s,k})+\rho_1(U_p^{s,k+1}-U_p^{s,k})+Z_p^{s,k+1}=0.
		\end{equation*}
		Using this normal vector in the definition of $r_p^{s,k+1}$, it obtains
		\begin{equation*}
			\begin{aligned}
				r_p^{s,k+1}
				&\leq\|\nabla_{U_p}f_p(U_p^{s,k+1},V^s)+Z_p^{s,k+1}\|_F^2\\
				&=\|\nabla_{U_p}f_p(U_p^{s,k+1},V^s)-\nabla_{U_p}f_p(U_p^{s,k+1},V_p^{s,k})-\rho_1(U_p^{s,k+1}-U_p^{s,k})\|_F^2\\
				&\leq2L_{1,p}^2\|V_p^{s,k}-V^s\|_F^2+2\rho_1^2\|\Delta U_p^{s,k}\|_F^2,
			\end{aligned}
		\end{equation*}
		where the last inequality follows from the cross Lipschitz continuity of $\nabla_{U_p}f_p$ and $\|A+B\|_F^2\leq2\|A\|_F^2+2\|B\|_F^2$. This completes the proof.
		
	\end{proof}  
	\end{appendices}
	

	
\end{document}